\documentclass[aap,MSNbibl,noautosecdot,dvips]{arximspdf}
\usepackage{mathbh}
\usepackage[lined,boxed,commentsnumbered]{algorithm2e}
\usepackage{graphicx}

\doi{10.1214/13-AAP975} 
\volume{24}
\issue{6}
\pubyear{2014}
\firstpage{2246}
\lastpage{2296}

\makeatletter
\let\my@algocf@latexcaption\algocf@latexcaption
\let\my@addcontentsline\addcontentsline
\long\def\algocf@latexcaption#1[#2]#3{%
\def\addcontentsline##1##2##3{}%
\my@algocf@latexcaption{#1}[#2]{#3}%
\global\let\addcontentsline\my@addcontentsline%
}
\newcommand{\rrvert}{\vert}
\newcommand{\llvert}{\vert}
\newtheorem{theorem}{Theorem}
\newtheorem{claim}[theorem]{Claim}
\newtheorem{corollary}[theorem]{Corollary}
\newtheorem{lemma}[theorem]{Lemma}
\newtheorem{proposition}[theorem]{Proposition}
\newproclaim{remark}[theorem]{Remark}
\makeatother

\begin{document}
\begin{frontmatter}

\title{Long runs under a conditional limit distribution}
\runtitle{Long runs under a conditional limit distribution}

\begin{aug}
\author{\fnms{Michel} \snm{Broniatowski}\corref{}\ead[label=e1]{michel.broniatowski@upmc.fr}}
\and
\author{\fnms{Virgile} \snm{Caron}\ead[label=e2]{virgile.caron@upmc.fr}}
\runauthor{M. Broniatowski and V. Caron}
\affiliation{Universit\'{e} Pierre Et Marie Curie, LSTA}

\address{LSTA\\
Universit\'{e} Pierre et Marie Curie---Paris 6\\
75005 Paris\\
France\\
\printead{e1}\\
\phantom{E-mail:\ }\printead*{e2}}
\end{aug}

\received{\smonth{2} \syear{2012}}
\revised{\smonth{9} \syear{2013}}


\begin{abstract}
This paper presents a sharp approximation of the density of long runs
of a
random walk conditioned on its end value or by an average of a function of
its summands as their number tends to infinity. In the large deviation range
of the conditioning event it extends the Gibbs conditional principle in the
sense that it provides a description of the distribution of the random walk
on long subsequences. An approximation of the density of the runs is also
obtained when the conditioning event states that the end value of the random
walk belongs to a thin or a thick set with a nonempty interior. The
approximations hold either in probability under the conditional distribution
of the random walk, or in total variation norm between measures. An application
of the approximation scheme to the evaluation of rare event probabilities
through importance sampling is provided. When the conditioning event is in
the range of the central limit theorem, it provides a tool for statistical
inference in the sense that it produces an effective way to implement the
Rao--Blackwell theorem for the improvement of estimators; it also leads to
conditional inference procedures in models with nuisance parameters. An
algorithm for the simulation of such long runs is presented, together with
an algorithm determining the maximal length for which the approximation is
valid up to a prescribed accuracy.
\end{abstract}

%


\end{frontmatter}

\section{\texorpdfstring{Context and scope.}{Context and scope}}\label{sec1}

This paper explores the asymptotic distribution of a random walk conditioned
on its final value as the number of summands increases. Denote $\mathbf
{X}_{1}^{n}:= ( \mathbf{X}_{1},\ldots,\mathbf{X}_{n} ) $  a set
of $n$ independent copies of a real random variable $\mathbf{X}$ with
density $p_{%
\mathbf{X}}$ on $\mathbb{R}$ and $\mathbf{S}_{1,n}:=\mathbf
{X}_{1}+\cdots+\mathbf{X}_{n}$. We consider approximations of the density of the
vector $\mathbf{X}_{1}^{k}= ( \mathbf{X}_{1},\ldots,\mathbf{X}%
_{k} ) $ on $\mathbb{R}^{k}$ when $\mathbf{S}_{1,n}= na _{n}$, and $a_{n}$
is a convergent sequence. The integer valued sequence $k:=k_{n}$ is such
that
%
\renewcommand{\theequation}{K\arabic{equation}}
\begin{equation}
0\leq\lim\sup_{n\rightarrow\infty}k/n\leq{1} 
\label{k/nTENDSTO1}
\end{equation}
together with
%
\begin{equation}
\lim_{n\rightarrow\infty}n-k=\infty. \label{pas trop vite}
\end{equation}
Therefore we may consider the asymptotic behavior of the density of the
trajectory of the random walk on long runs. For the sake of
applications we also
address the case when $\mathbf{S}_{1,n}$ is substituted by $\mathbf{U}%
_{1,n}:=u ( \mathbf{X}_{1} ) +\cdots+u ( \mathbf{X}_{1}
) $
for some real valued measurable function $u$, and when the conditioning
event is $ ( \mathbf{U}_{1,n}=u_{1,n} ) $ where $u_{1,n}/n$
converges as $n$ tends to infinity. A complementary result provides an
estimation for the case when the conditioning event is a large set in the
large deviation range, $ ( \mathbf{U}_{1,n}\in nA ) $ where
$A$ is
a Borel set with nonempty interior with $Eu(\mathbf{X}) < \mathtt
{essinf} A$;
two cases are considered, according to the local dimension of $A$ at its
essential infimum point \texttt{essinf}$A$.

The interest in this question stems from various sources. When $k$ is fixed
(typically $k=1$) this is a version of the \textit{Gibbs conditional
principle} which has been studied extensively for fixed $a_{n}\neq
E\mathbf{%
X}$, therefore under a \textit{large deviation} condition. Diaconis and
Freedman \cite%
{DiaconisFreedman1988} have considered this issue also in the case $%
k/n\rightarrow\theta$ for $0\leq\theta<1$, in connection with de
Finetti's theorem for exchangeable finite sequences. Their interest was
related to the approximation of the density of $\mathbf{X}_{1}^{k}$ by the
\textit{product density} of the summands $\mathbf{X}_{i}$'s, and
therefore on
the validity of the independence of the $\mathbf{X}_{i}$'s under
conditioning. Their result is in the spirit of Van Camperhout and Cover
\cite{VanCamperhoutCover1981},
and parallels can be drawn with Csisz{\'a}r's \cite{Csiszar1084}
asymptotic conditional
independence result, when the conditioning event is $ ( \mathbf{S}%
_{1,n}> na _{n} ) $ with $a_{n}$ fixed and larger than $E\mathbf
{X}$. In
the same vein and under the same \textit{large deviation} condition
Dembo and Zeitouni \cite%
{DemboZeitouni1996} considered similar problems. This question is also of
importance in statistical physics. Numerous papers pertaining to structural
properties of polymers deal with this issue, and we refer to \cite%
{denHollanderWeiss1988} and \cite{denHollanderWeiss1988a} for a description
of those problems and related results. In the moderate deviation case,
Ermakov \cite%
{ERmakov2006} also considered a similar problem when $k=1$.

The approximation of conditional densities is the basic ingredient for the
numerical estimation of integrals through improved Monte Carlo techniques.
Rare event probabilities may be evaluated through importance sampling
techniques; efficient sampling schemes consist of the simulation of random
variables under a proxy of a conditional density, often with respect to
conditioning events of the form $ ( \mathbf{U}_{1,n}> na _{n} ) $;
optimizing these schemes has been a motivation for this work.

In parametric statistical inference, conditioning on the observed value
of a
statistic leads to a reduction of the mean square error of some
estimate of
the parameter; the famous Rao--Blackwell and Lehmann--Scheff\'{e} theorems
can be implemented when a simulation technique produces samples
according to
the distribution of the data conditioned on the value of some observed
statistics. In these applications the conditioning event is local, and when
the statistic is of the form $\mathbf{U}_{1,n}$, then the observed
value $%
u_{1,n}$ satisfies $\lim_{n\rightarrow\infty}u_{1,n}/n=Eu (
\mathbf{X}%
)$. Such is the case in exponential families when $\mathbf{U}_{1,n}$
is a sufficient statistic for the parameter. Other fields of applications
pertain to parametric estimation where conditioning by the observed
value of
a sufficient statistic for a nuisance parameter produces optimal inference
via maximum likelihood in the conditioned model. In general this
conditional density is unknown; the approximation produced in this paper
provides a tool for the solution of these problems.

For both importance sampling and for the improvement of estimators, the
approximation of the conditional density of $\mathbf{X}_{1}^{k}$ on long
runs should be of a special form: it has to be a density on $\mathbb{R}%
^{k}$, easy to simulate, and the approximation should be sharp. For these
applications the relative error of the approximation should be small on the
simulated paths only. Also for inference via maximum likelihood under
nuisance parameters the approximation has to be accurate on the sample itself
and not on the entire space.

Our first set of results provides a very sharp approximation scheme;
numerical evidence on exponential runs with length $n=1000$ provide a
\textit{relative error} of the approximation of order less than 100\% for
the density of the first 800 terms when evaluated on the sample paths
themselves, thus on the significant part of the support of the conditional
density; this very sharp approximation rate is surprising in such a large
dimensional space, and it illustrates the fact that the conditioned measure
occupies a very small part of the entire space. Therefore the approximation
of the density of $\mathbf{X}_{1}^{k}$ is not performed on the sequence of
entire spaces $\mathbb{R}^{k}$, but merely on a sequence of subsets of $
\mathbb{R}^{k}$ which contain the trajectories of the conditioned
random walk
with probability going to $1$ as $n$ tends to infinity; the
approximation is
performed on \textit{typical paths}.

The extension of our results from typical paths to the whole space
$\mathbb{R%
}^{k}$ holds: convergence of the relative error on large sets imply
that the
total variation distance between the conditioned measure and its
approximation goes to $0$ on the entire space. So our results provide an
extension of Diaconis and Freedman \cite{DiaconisFreedman1988} and
Dembo and Zeitouni \cite{DemboZeitouni1996} who
considered the case when $k$ is of small order with respect to $n$; the
conditions which are assumed in the present paper are weaker than those
assumed in the previously cited works; however, in contrast with their
results, we
do not provide explicit rates for the convergence to $0$ of the total
variation distance on $\mathbb{R}^{k} $.

It would have been of interest to consider sharper convergence criteria than
the total variation distance; the $\chi^{2}$-distance, which is the mean
square relative error, cannot be bounded through our approach on the entire
space $\mathbb{R}^{k}$, since it is only suitable for large sets of
trajectories (whose probability goes to $1$ as $n$ increases); this is not
sufficient to bound its expected value under the conditional sampling.

This paper is organized as follows. Section~\ref{sec2} presents the approximation
scheme for the conditional density of $\mathbf{X}_{1}^{k}$ under the
conditioning point sequence $ ( \mathbf{S}_{1,n}= na _{n} )$. In Section~\ref{sec3}, it is extended to the case when the conditioning family of events is
written as $ ( \mathbf{U}_{1,n}=u_{1,n} )$. The value of $k$
for which this
approximation is appropriate is discussed; an algorithm for the
implementation of
this rule is proposed. Algorithms for the simulation of random variables
under the approximating scheme are also presented. Section~\ref{sec4} extends the
results of Section~\ref{sec3} when conditioning on large sets. Two applications are
presented in Section~\ref{sec5}; the first one pertains to Rao--Blackwellization of
estimators, hence on the application of the results of Section~\ref{sec3} when the
conditioning point is such that $\lim_{n\rightarrow\infty
}u_{1,n}/n=Eu (
\mathbf{X} ) $; in the second application the result of Section~\ref{sec4} is
used to derive small variance estimators of rare event probabilities through
importance sampling; in this case the conditioning event is in the
range of
the large deviation scale.

The main steps of the proofs are in the core of the paper; some of the
technicalities are left to the \hyperref[app]{Appendix}.

\section{\texorpdfstring{Random walks conditioned on their sum.}
{Random walks conditioned on their sum}}\label{sec2}

\subsection{\texorpdfstring{Notation and hypothesis.}
{Notation and hypothesis}}\label{sec2.1}

In this section the conditioning point event is written as
\[
\mathcal{E}_{n}:= ( \mathbf{S}_{1,n}= na _{n} ).
\]

We assume that $\mathbf{X}$ satisfies the Cram\'er condition; that is,
$\mathbf{X}$
has a finite moment generating function $\Phi(t):=E[\exp (t\mathbf
{X} )]$ in a nonempty neighborhood of $0$. Denote
\begin{eqnarray*}
m(t)&:=&\frac{d}{dt}\log\Phi(t),
\\
s^{2}(t)&:=&\frac{d}{dt}m(t),
\\
\mu_{3}(t)&:=&\frac{d}{dt}s^{2}(t).
\end{eqnarray*}
The values of $m(t)$, $s^{2}$ and $\mu_{3}(t)$ are the expectation, the
variance and the kurtosis of the \textit{tilted} density
%
\renewcommand{\theequation}{\arabic{equation}}
\setcounter{equation}{0}
\begin{equation}
\pi^{\alpha}(x):=\frac{\exp(tx)}{\Phi(t)}p(x), \label{tilted density}
\end{equation}
where $t$ is the unique solution of the equation $m(t)=\alpha$ when
$\alpha$
belongs to the support of $\mathbf{X}$. Conditions on $\Phi(t)$ which
ensure existence and uniqueness of $t$ are referred to as \textit{steepness
properties; }we refer to \cite{Barndorff1978}, page 153 ff., for all
properties of moment generating functions used in this paper. Denote
$\Pi
^{\alpha}$ the probability measure with density $\pi^{\alpha}$.

We also assume that the characteristic function of $\mathbf{X}$ is in $L^{r}$
for some $r\geq1$ which is necessary for the Edgeworth expansions to be
performed.\vadjust{\goodbreak}

The probability measure of the random vector $\mathbf{X}_{1}^{n}$ on $%
\mathbb{R}^{n}$ conditioned upon $\mathcal{E}_{n}$ is denoted $P_{ na _{n}}$.
We also denote $P_{ na _{n}}$ the corresponding distribution of $\mathbf
{X}%
_{1}^{k}$ conditioned upon $\mathcal{E}_{n}$; the vector $\mathbf{X}_{1}^{k}$
then has a density with respect to the Lebesgue measure on $\mathbb{R}^{k}$
for $1\leq k<n$, which will be denoted $p_{ na _{n}}$. For a general r.v.
$\mathbf{Z}$ with density $p$, $p ( \mathbf{Z}=z ) $ denotes the
value of $p$ at point $z$. Hence, $p_{ na _{n}} ( x_{1}^{k} )
=p ( \mathbf{X}_{1}^{k}=x_{1}^{k}|\mathbf{S}_{1,n}= na _{n} ) $. The
normal density function on $\mathbb{R}$ with mean $\mu$ and variance
$\tau$
at $x$ is denoted $\mathfrak{n} ( \mu,\tau,x ) $. When $\mu=0$
and $\tau=1$, the standard notation $\mathfrak{n} ( x ) $ is used.

\subsection{\texorpdfstring{A first approximation result.}
{A first approximation result}}\label{sec2.2}

We first put forward a simple result which provides an approximation of the
density $p_{ na _{n}}$ of the measure $P_{ na _{n}}$ on $\mathbb{R}^{k}$
when $k$
satisfies (\ref{k/nTENDSTO1}) and (\ref{pas trop vite}). For $i\leq j$
denote
\[
s_{i,j}:=x_{i}+\cdots+x_{j}.
\]
Denote $a:=a_{n}$ omitting the index $n$ for clarity.

We make use of the following property which states the invariance of
conditional densities under tilting: For $1\leq i\leq j\leq n$, for all
$%
a$ in the range of $\mathbf{X}$, for all~$u$ and $s$
%
\begin{equation}
p (  \mathbf{S}_{i,j}=u\rrvert \mathbf{S}_{1,n}=s ) =
\pi ^{a} (  \mathbf{S}_{i,j}=u\rrvert
\mathbf{S}_{1,n}=s ), \label{inv tilting}
\end{equation}
where $\mathbf{S}_{i,j}:=\mathbf{X}_{i}+\cdots+\mathbf{X}_{j}$ together
with $%
\mathbf{S}_{1,0}=s_{1,0}=0$. By the Bayes formula it holds that
%
\begin{eqnarray}
\label{Bayes1}p_{ na } \bigl( x_{1}^{k} \bigr) &=&\prod
_{i=0}^{k-1}p (  \mathbf {X}%
_{i+1}=x_{i+1}\rrvert \mathbf{S}_{i+1,n}= na -s_{1,i}
)
\\
&=&\prod_{i=0}^{k-1}\pi^{a}(
\mathbf{X}_{i+1}=x_{i+1})\frac{\pi
^{a} ( \mathbf{S}_{i+2,n}= na -s_{1,i+1} ) }{\pi^{a} (
\mathbf{S}%
_{i+1,n}= na -s_{1,i} ) }
\nonumber
\\
&=& \Biggl[ \prod_{i=0}^{k-1}
\pi^{a}(\mathbf{X}_{i+1}=x_{i+1}) \Biggr]
\frac{\pi^{a} ( \mathbf{S}_{k+1,n}= na -s_{1,k} ) }{\pi
^{a} (
\mathbf{S}_{1,n}= na  ) }. \label{Bayes2}
\end{eqnarray}
Denote $\overline{\mathbf{S}_{k+1,n}}$ and $\overline{\mathbf
{S}_{1,n}}$ the
normalized versions of $\mathbf{S}_{k+1,n}$ and $\mathbf{S}_{1,n}$
under the
sampling distribution $\Pi^{a}$. By (\ref{Bayes2})
\begin{eqnarray*}
p_{ na } \bigl( x_{1}^{k} \bigr) &=& \Biggl[ \prod
_{i=0}^{k-1}\pi^{a}(%
\mathbf{X}_{i+1}=x_{i+1}) \Biggr]\\
&&{}\times \frac{\sqrt{n}}{\sqrt{n-k}}
\frac{\pi
^{a} ( \overline{\mathbf{S}_{k+1,n}}={(ka-s_{1,k})}/{(s_{a}\sqrt {n-k})}%
) }{\pi^{a} ( \overline{\mathbf{S}_{1,n}}=0 ) }.
\end{eqnarray*}
A first order Edgeworth expansion is performed in both terms of the
ratio in
the above display; see Remark \ref{Remark Edgeworth array} below. This
yields, assuming (\ref{k/nTENDSTO1}) and (\ref{pas trop vite}), the following:

\begin{proposition}
\label{PropResultatsurRkProposition} For all $x_{1}^{k}$ in $\mathbb
{R}^{k}$
%
\begin{eqnarray}\label{PropResultatsurRk}
p_{ na } \bigl( x_{1}^{k} \bigr) &=& \Biggl[ \prod
_{i=0}^{k-1}\pi^{a}(%
\mathbf{X}_{i+1}=x_{i+1}) \Biggr]\nonumber\\
&&{}\times \biggl[ \frac{\mathfrak{n} ( {
(ka-s_{1,k})}/{(s(t^{a})\sqrt{n-k})} ) }{\mathfrak{n} ( 0 )
}
\sqrt{%
\frac{n}{n-k}}
\\
&& \hspace*{-16pt}\qquad\quad{}\times\biggl( 1+\frac{\mu_{3}(t^{a})}{6s^{3}(t^{a})\sqrt {n-k}}H_{3} \biggl( \frac{ka-s_{1,k}}{s(t^{a})\sqrt{n-k}}
\biggr) \biggr) +O \biggl( \frac
{1}{\sqrt{%
n}} \biggr) \biggr],
\nonumber
\end{eqnarray}
where $H_{3}(x):=x^{3}-3x$. The value of $t^{a}$ is defined through $%
m(t^{a})=a$.
\end{proposition}

Despite its appealing aspect, (\ref{PropResultatsurRk}) is of poor
value for
applications, since it does not yield an explicit way to simulate samples
under a proxy of $p_{ na }$ for large values of $k$. The other way is to
construct the approximation of $p_{ na }$ step by step, approximating the
terms in (\ref{Bayes1}) one by one and using the invariance under the
tilting at each step, which introduces a product of different tilted
densities in (\ref{Bayes2}). This method produces a valid approximation
of $%
p_{ na }$ on subsets of $\mathbb{R}^{k}$ which contain the trajectories of
the conditioning random walk with larger and larger probability, going
to $1$
as $n$ tends to infinity.

This introduces the main focus of this paper.

\subsection{\texorpdfstring{A recursive approximation scheme.}
{A recursive approximation scheme}}\label{sec2.3}

We introduce a positive sequence $\varepsilon_{n}$ which satisfies
%
\renewcommand{\theequation}{E\arabic{equation}}
\setcounter{equation}{0}
\begin{eqnarray}
\label{E1} \lim_{n\rightarrow\infty}\varepsilon_{n}\sqrt{n-k} & =&
\infty,
\\
\lim_{n\rightarrow\infty}\varepsilon_{n} ( \log n )
^{2} & =&0. \label{E2}
\end{eqnarray}

It will be shown that $\varepsilon_{n} ( \log n ) ^{2}$ is the
rate of
accuracy of the approximating scheme.

We denote $a$ the generic term of the convergent sequence $ (
a_{n} ) _{n\geq1}$. For clarity the dependence on $n$ of all
quantities involved in the subsequent development is omitted in the notation.

\subsubsection{\texorpdfstring{Approximation of the density of the runs.}
{Approximation of the density of the runs}}\label{sec2.3.1}

Define a density $g_{ na }(y_{1}^{k})$ on $\mathbb{R}^{k}$ as follows. Set
\[
g_{0}( y_{1}\rrvert y_{0}):=
\pi^{a}(y_{1})
\]
with $y_{0}$ arbitrary, and for $1\leq i\leq k-1$ define $g(
y_{i+1}\rrvert  y_{1}^{i})$ recursively.

Set $t_{i}$ to be the unique solution of the equation
%
\renewcommand{\theequation}{\arabic{equation}}
\setcounter{equation}{5}
\begin{equation}
m_{i}:=m(t_{i})=\frac{n}{n-i} \biggl( a-
\frac{s_{1,i}}{n} \biggr), \label{mi_centre}
\end{equation}
where $s_{1,i}:=y_{1}+\cdots+y_{i}$. The tilted adaptive family of
densities $%
\pi^{m_{i}}$ is the basic ingredient of the derivation of approximating
scheme$.$ Let
\[
s_{i}^{2}:=\frac{d^{2}}{dt^{2}} \bigl( \log E_{\pi^{m_{i}}}
\exp(t\mathbf {X}) \bigr) ( 0 )
\]
and
\[
\mu_{j}^{i}:=\frac{d^{j}}{dt^{j}} \bigl( \log
E_{\pi^{m_{i}}}\exp (t\mathbf{X}) \bigr) ( 0 ), \qquad j=3,4,
\]
which are the second, third and fourth cumulants of $\pi^{m_{i}}$. Let
%
\begin{equation}
g\bigl( y_{i+1}\rrvert y_{1}^{i}
\bigr)=C_{i}p_{\mathbf
{X}}(y_{i+1})\mathfrak{%
n} ( \alpha\beta+a,\beta,y_{i+1} ) \label{g-i}
\end{equation}
be a density where
%
\begin{eqnarray}
\alpha&=&t_{i}+\frac{\mu_{3}^{i}}{2s_{i}^{2} ( n-i-1 ) } \label{beta},
\\
\beta&=&s_{i}^{2} ( n-i-1 ) \label{alpha}
\end{eqnarray}
and $C_{i}$ is a normalizing constant.

Define
%
\begin{equation}
g_{ na }\bigl(y_{1}^{k}\bigr):=g_{0}(
 y_{1}\rrvert y_{0}) \prod
_{i=1}^{k-1}g\bigl( y_{i+1}\rrvert
y_{1}^{i}\bigr). \label{g_a}
\end{equation}

We then have:

\begin{theorem}
\label{Thm:Approx_local_cond_density} Assume  (\ref{k/nTENDSTO1}) and
(\ref%
{pas trop vite}) together with (\ref{E1}) and (\ref{E2}). Let
$Y_{1}^{n}$ be
a sample from density $p_{ na }$. Then
%
\begin{eqnarray}\label{local approx under exact value}
p_{ na } \bigl( Y_{1}^{k} \bigr)&:=&p\bigl(
\mathbf {X}_{1}^{k}=Y_{1}^{k}
\vert\mathbf{S}_{1,n}= na \bigr)
\nonumber
\\[-8pt]
\\[-8pt]
\nonumber
&=& g_{ na }
\bigl(Y_{1}^{k}\bigr) \bigl(1+o_{P_{ na }}\bigl(\varepsilon
_{n} ( \log n ) ^{2}\bigr)\bigr).
\end{eqnarray}
\end{theorem}

\begin{pf}
The proof uses Bayes's formula to write $p( \mathbf{X}%
_{1}^{k}=Y_{1}^{k}\rrvert \mathbf{S}_{1,n}= na )$ as a product of $k$
conditional densities of the individual terms of the trajectory
evaluated at $%
Y_{1}^{k}$. Each term of this product is approximated by an Edgeworth
expansion which together with the properties of $Y_{1}^{k}$ under $P_{ na }$
completes the proof. This proof is rather long, and we have deferred its
technical steps to the \hyperref[app]{Appendix}.

Denote $S_{1,0}=0$ and $S_{1,i}:=S_{1,i-1}+Y_{i}$. It holds that
%
\begin{eqnarray}
\label{joint density}
&&p\bigl( \mathbf{X}_{1}^{k}=Y_{1}^{k}
\rrvert \mathbf{S}_{1,n} = na \bigr)=p( \mathbf{X}_{1}=Y_{1}
\rrvert \mathbf{S}_{1,n}= na ),
\nonumber\\
&&\prod_{i=1}^{k-1}p\bigl(
\mathbf{X}_{i+1}=Y_{i+1}\rrvert \mathbf{X} 
_{1}^{i} =Y_{1}^{i},
\mathbf{S}_{1,n}= na \bigr)
\\
&&\qquad =\prod_{i=0}^{k-1}p (
\mathbf{X}_{i+1}=Y_{i+1}\rrvert \mathbf{S}_{i+1,n}= na -S_{1,i}
)
\nonumber
\end{eqnarray}
by independence of the r.v.'s $\mathbf{X}_{i}$'s.

Define $t_{i}$ through
\[
m(t_{i})=\frac{n}{n-i} \biggl( a-\frac{S_{1,i}}{n} \biggr)
\]
a function of the past r.v.'s $Y_{1}^{i}$, and set $m_{i}:=m(t_{i})$
and $%
s_{i}^{2}:=s^{2}(t_{i})$. By (\ref{inv tilting})
\begin{eqnarray*}
&& p (  \mathbf{X}_{i+1}=Y_{i+1}\rrvert
\mathbf{S}%
_{i+1,n}= na -S_{1,i} )
\\
&&\qquad =\pi^{m_{i}} \bigl(  \mathbf{X}_{i+1}=Y_{i+1}
\rrvert \mathbf {S}%
_{i+1}^{n}= na -S_{1,i}
\bigr)
\\
&&\qquad =\pi^{m_{i}} ( \mathbf{X}_{i+1}=Y_{i+1} )
\frac{\pi
^{m_{i}} ( \mathbf{S}_{i+2,n}= na -S_{1,i+1} ) }{\pi
^{m_{i}} (
\mathbf{S}_{i+1,n}= na -S_{1,i} ) },
\end{eqnarray*}
where we used the independence of the $\mathbf{X}_{j}$'s under $\pi
^{m_{i}}$. A precise evaluation of the dominating terms in this latest
expression is needed in order to handle the product~(\ref{joint density}).

Under the sequence of densities $\pi^{m_{i}}$ the i.i.d. r.v.'s
$\mathbf{X}%
_{i+1},\dots,\mathbf{X}_{n}$ define a triangular array which satisfies
a local
central limit theorem, and an Edgeworth expansion. Under $\pi
^{m_{i}}$, $%
\mathbf{X}_{i+1}$ has expectation $m_{i}$ and variance $s_{i}^{2}$. Center
and normalize both the numerator and denominator in the fraction which
appear in the last display. Denote $\overline{\pi_{n-i-1}}$ the
density of
the normalized sum $ ( \mathbf{S}_{i+2,n}-(n-i-1)m_{i} ) / (
s_{i}\sqrt{n-i-1} ) $ when the summands are i.i.d. with common
density $%
\pi^{m_{i}}$. Accordingly $\overline{\pi_{n-i}}$ is the density of the
normalized sum $ ( \mathbf{S}_{i+1,n}-(n-i)m_{i} ) / (
s_{i}%
\sqrt{n-i} ) $ under i.i.d. $\pi^{m_{i}}$ sampling. Hence, evaluating
both $\overline{\pi_{n-i-1}}$ and its normal approximation at point $%
Y_{i+1}$,
%
\begin{eqnarray}\label{condTilt}
&& p (  \mathbf{X}_{i+1}=Y_{i+1}\rrvert
\mathbf{S}%
_{i+1,n}= na -S_{1,i} )
\nonumber\\
&&\qquad =\frac{\sqrt{n-i}}{\sqrt{n-i-1}}\pi^{m_{i}} ( \mathbf{X}%
_{i+1}=Y_{i+1} ) \frac{\overline{\pi_{n-i-1}} (  (
m_{i}-Y_{i+1} ) /s_{i}\sqrt{n-i-1} ) }{\overline{\pi_{n-i}}(0)}
\\
&&\qquad:=\frac{\sqrt{n-i}}{\sqrt{n-i-1}}\pi^{m_{i}} ( \mathbf{X}%
_{i+1}=Y_{i+1} ) \frac{N_{i}}{D_{i}}.
\nonumber
\end{eqnarray}
The sequence of densities $\overline{\pi_{n-i-1}}$ converges pointwise to
the standard normal density under (\ref{E1}) which implies that $n-i$ tends
to infinity for all $1\leq i\leq k$, and an Edgeworth expansion to order
5 is performed for the numerator and the denominator. The main arguments
used in order to obtain the order of magnitude of the involved quantities
are (i) a maximal inequality which controls the magnitude of~$m_{i}$ for
all $i$ between $0$ and $k-1$ (Lemma \ref{LemmaMaxm_in}), (ii) the order
of the maximum of the~$Y_{i}$'s (Lemma \ref%
{Lemma_max_X_i_under_conditioning}). As proved in the \hyperref[app]{Appendix},
%
\begin{equation}
N_{i}=\mathfrak{n} ( -Y_{i+1}/s_{i}\sqrt{n-i-1}
) \cdot A \cdot B+O_{P_{ na }} \biggl( \frac{1}{ ( n-i-1 )
^{3/2}} \biggr),
\label{num approx fixed i}
\end{equation}
where
%
\begin{equation}
A:= \biggl( 1+\frac{aY_{i+1}}{s_{i}^{2}(n-i-1)}-\frac
{a^{2}}{2s_{i}^{2}(n-i-1)}%
+
\frac{o_{P_{ na }}(\varepsilon_{n}\log n)}{n-i-1} \biggr) \label{Adem}
\end{equation}
and
%
\begin{equation}
B:= \pmatrix{\displaystyle 1-\frac{\mu_{3}^{i}}{2s_{i}^{4} ( n-i-1 ) }(a-Y_{i+1})
\vspace*{2pt}\cr
\displaystyle-\frac{\mu_{3}^{i}-s_{i}^{4}}{8s_{i}^{4}(n-i-1)}-\frac{15(\mu
_{3}^{i})^{2}%
}{72s_{i}^{6}(n-i-1)}+\frac{O_{P_{ na }} ( (\log n)^{2} ) }{ (
n-i-1 ) ^{2}}} \label{Bdem}.
\end{equation}
The $O_{P_{ na }} ( \frac{1}{ ( n-i-1 ) ^{3/2}} ) $
term in (%
\ref{num approx fixed i}) is uniform on $ (m_{i}-Y_{i+1}
)/s_{i}%
\sqrt{n-i-1}$. Turn back to (\ref{condTilt}) and perform the same Edgeworth
expansion in the denominator, which is written as
%
\begin{equation}
D_{i}=\mathfrak{n} ( 0 ) \biggl( 1-\frac{\mu
_{4}^{i}-3s_{i}^{4}}{%
8s_{i}^{4}(n-i)}-
\frac{15(\mu_{3}^{i})^{2}}{72s_{i}^{6}(n-i)} \biggr) +O_{P_{ na }} \biggl( \frac{1}{ ( n-i ) ^{3/2}} \biggr).
\label{PI 0}
\end{equation}
The terms in $g( Y_{i+1}\rrvert  Y_{1}^{i})$ follow from an
expansion in the ratio of the two expressions (\ref{num approx fixed
i}) and
(\ref{PI 0}) above. The Gaussian contribution is explicit in (\ref{num
approx fixed i}) while the term $\exp (\frac{\mu
_{3}^{i}}{2s_{i}^{4} (
n-i-1 ) }Y_{i+1} )$ is the dominant term in $B$. Turning to
(\ref%
{condTilt}) and comparing with (\ref{local approx under exact value}) it
appears that the normalizing factor $C_{i}$ in $g( Y_{i+1}\rrvert
Y_{1}^{i})$ compensates the term $\frac{\sqrt{n-i}}{\Phi(t_{i})\sqrt {n-i-1}}%
\exp ( \frac{-a\mu_{3}^{i}}{2s_{i}^{2}(n-i-1)} )$, where the
term $\Phi(t_{i})$ comes from $\pi^{m_{i}} ( \mathbf{X}%
_{i+1}=Y_{i+1} ) $. Furthermore the product of the remaining terms
in the
above approximations in (\ref{num approx fixed i}) and (\ref{PI 0})
form the $1+o_{P_{ na }} ( \varepsilon_{n} ( \log n )
^{2} ) $
approximation rate, as claimed$.$ Details are deferred to the \hyperref[app]{Appendix}. This
yields %
\[
p\bigl( \mathbf{X}_{1}^{k}=Y_{1}^{k}
\rrvert \mathbf {S}_{1,n}= na \bigr)= \bigl( 1+o_{P_{ na }} \bigl(
\varepsilon_{n} ( \log n ) ^{2} \bigr) \bigr) g_{0}(
 Y_{1}\rrvert Y_{0})\prod
_{i=1}^{k-1}g\bigl( Y_{i+1}\rrvert
Y_{1}^{i}\bigr),
\]
which completes the proof of the theorem.
\end{pf}

That the variation distance between $P_{ na _{n}}$ and $G_{ na _{n}}$
tends to $0$ as $n\rightarrow\infty$ is stated in Section~\ref{sec3}.

\begin{figure}

\includegraphics{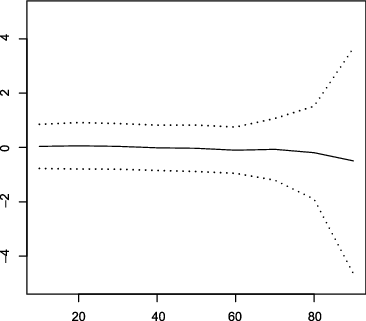}

\caption{$\overline{\operatorname{ERE}}(k)$ (solid line) along with upper and lower
bound of
$\overline{\operatorname{CI}}(k)$ (dotted line) as a function of $k$ with $n=100$ and $a$
such that $P_{n}\simeq10^{-8}$.}
\label{table:choix_de_k_GD_Approx_100}
\end{figure}

\begin{figure}[b]

\includegraphics{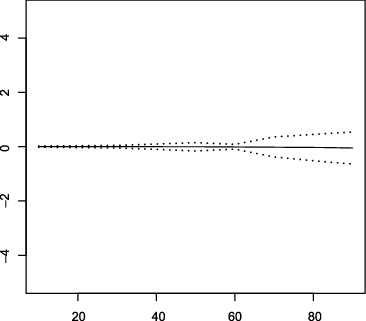}

\caption{$\operatorname{ERE}(k)$ (solid line) along with upper and lower bound of
$\operatorname{CI}(k)$
(dotted line) as a function of $k$ with $n=100$ and $a$ such that $%
P_{n}\simeq10^{-8}$.}
\label{table:choix_de_k_GD_Vrai_100}
\end{figure}

\begin{remark}
When the $\mathbf{X}_{i}$'s are i.i.d. with a standard normal density, then
the result in the above approximation theorem holds with $k=n-1$ implying
that $p( \mathbf{X}_{1}^{n-1}=x_{1}^{n-1}\rrvert \mathbf{S}%
_{1,n}= na )=g_{ na } ( x_{1}^{n-1} ) $ for all $x_{1}^{n-1}$ in $%
\mathbb{R}^{n-1}$. This extends to the case when they have an infinitely
divisible distribution. However, formula (\ref{local approx under exact
value}%
) holds true without the error term only in the Gaussian case. Similar exact
formulas can be obtained for infinitely divisible distributions using
(\ref%
{joint density}) where no use of tilting is made. Such formulas are
used to produce
Figures~\ref{table:choix_de_k_GD_Approx_100}, \ref%
{table:choix_de_k_GD_Vrai_100}, \ref{table:choix_de_k_GD_Approx_1000}
and %
\ref{table:choix_de_k_GD_Vrai_1000} in order to assess the validity of the
selection rule for $k$ in the exponential case.
\end{remark}

\begin{remark}
The density in (\ref{g-i}) is a slight modification of $\pi^{m_{i}}$. The
modification from $\pi^{m_{i}} (y_{i+1} )$ to $g
(y_{i+1}\vert
y_{1}^{i} )$ is a small shift in the location parameter depending both
on $a$ and on the skewness of $p$, and a change in the variance: large
values of $\mathbf{X}_{i+1}$ have smaller weight for large $i$, so that the
distribution of $\mathbf{X}_{i+1}$ tends to concentrate around $m_{i}$
as $i$
approaches $k$.
\end{remark}

\begin{remark}
\label{Remark Edgeworth array} In Theorem \ref{Thm:Approx_local_cond_density},
as in Proposition \ref{PropResultatsurRkProposition}, Theorem \ref
{ThmApproxsousf(x)} or Lemma \ref{Lemma_max_X_i_under_conditioning}, we
use an Edgeworth expansion for the density of the normalized sum of the
$(n-i)$th row of some triangular array of row-wise independent r.v.'s
with a common density. Consider the i.i.d. r.v.'s $\mathbf
{X}_{1},\ldots,\mathbf{X}_{n}$
with common density $\pi^{a}(x)$ where $a$\vadjust{\goodbreak} may depend on $n$ but remains
bounded. The Edgeworth expansion with respect to the normalized
density of $%
\mathbf{S}_{1,n}$ under $\pi^{a}$ can be derived following closely the
proof given, for example, in \cite{Feller1971}, page 532 ff., by
substituting the cumulants of $p$ by those of $\pi^{a}$. Denote
$\varphi
_{a}(z)$ the characteristic function of $\pi^{a}(x)$. Clearly for any $
\delta>0$ there exists $q_{a,\delta}<1$ such that $\llvert \varphi
_{a}(z)\rrvert <q_{a,\delta}$ and since $a$ is bounded, $%
\sup_{n}q_{a,\delta}<1$. Therefore inequality (2.5) in \cite{Feller1971},
page 533 holds. With $\psi_{n}$ defined as in \cite{Feller1971}, (2.6) holds
with $\varphi$ replaced by $\varphi_{a}$ and $\sigma$ by $s(t^{a})$;
(2.9) holds, which completes the proof of the Edgeworth expansion in the
simple case. The proof is analogous for higher order expansions.
\end{remark}

\subsubsection{\texorpdfstring{Sampling under the approximation.}
{Sampling under the approximation}}\label{sec2.3.2}
%
\begin{figure}

\includegraphics{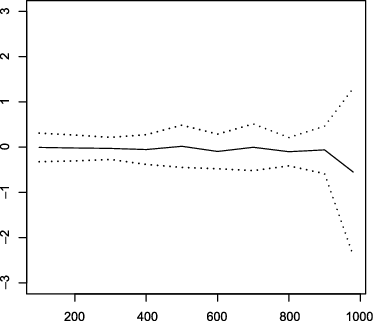}

\caption{$\overline{\operatorname{ERE}}(k)$ (solid line) along with upper and lower
bound of
$\overline{\operatorname{CI}}(k)$ (dotted line) as a function of $k$ with $n=1000$ and $a$
such that $P_{n}\simeq10^{-8}$.}
\label{table:choix_de_k_GD_Approx_1000}
\end{figure}
%
\begin{figure}[b]

\includegraphics{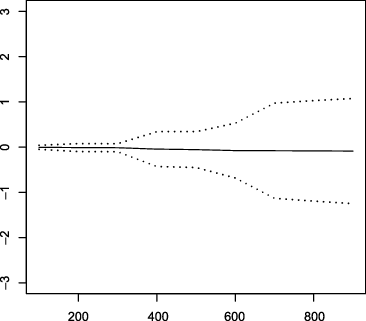}

\caption{$\operatorname{ERE}(k)$ (solid line) along with upper and lower bound of
$\operatorname{CI}(k)$
(dotted line) as a function of $k$ with $n=1000$ and $a$ such that $%
P_{n}\simeq10^{-8}$.}
\label{table:choix_de_k_GD_Vrai_1000}
\end{figure}
Applications of Theorem \ref{Thm:Approx_local_cond_density} in importance
sampling procedures and in Statistics require a reverse result. So assume
that $Y_{1}^{k}$ is a random vector generated under $G_{ na }$ with
density $%
g_{ na }$. Can we state that $g_{ na } ( Y_{1}^{k} ) $ is a good
approximation for $p_{ na } ( Y_{1}^{k} ) $? This holds true. We
state a simple lemma in this direction.

Let $\mathfrak{R}_{n}$ and $\mathfrak{S}_{n}$ denote two p.m.'s on
$\mathbb{R}%
^{n}$ with respective densities $\mathfrak{r}_{n}$ and~$\mathfrak{s}_{n}$.

\begin{lemma}
\label{Lemma:commute_from_p_n_to_g_n} Suppose that for some sequence $%
\varepsilon_{n}$ which tends to $0$ as $n$ tends to infinity
%
\begin{equation}
\mathfrak{r}_{n} \bigl( Y_{1}^{n} \bigr) =
\mathfrak{s}_{n} \bigl( Y_{1}^{n} \bigr) \bigl(
1+o_{\mathfrak{R}_{n}}(\varepsilon_{n}) \bigr) \label{p_n equiv g_n under p_n}
\end{equation}
as $n$ tends to $\infty$. Then
%
\begin{equation}
\mathfrak{s}_{n} \bigl( Y_{1}^{n} \bigr) =
\mathfrak{r}_{n} \bigl( Y_{1}^{n} \bigr) \bigl(
1+o_{\mathfrak{S}_{n}}(\varepsilon_{n}) \bigr). \label{g_n equiv p_n under g_n}
\end{equation}
\end{lemma}

\begin{pf}
Denote
\[
A_{n,\varepsilon_{n}}:= \bigl\{ y_{1}^{n}\dvtx (1-
\varepsilon_{n})\mathfrak{s} 
_{n} \bigl(
y_{1}^{n} \bigr) \leq\mathfrak{r}_{n} \bigl(
y_{1}^{n} \bigr) \leq%
\mathfrak{s}_{n}
\bigl( y_{1}^{n} \bigr) (1+\varepsilon_{n}) \bigr\}
.
\]
It holds for all positive $\delta$,
\[
\lim_{n\rightarrow\infty}\mathfrak{R}_{n} ( A_{n,\delta\varepsilon_{n}} ) =1.
\]
Write
\[
\mathfrak{R}_{n} ( A_{n,\delta\varepsilon_{n}} ) =\int\mathbf {1}%
_{A_{n,\delta\varepsilon_{n}}} \bigl( y_{1}^{n} \bigr) \frac{\mathfrak{r}
_{n} ( y_{1}^{n} ) }{\mathfrak{s}_{n}(y_{1}^{n})}
\mathfrak{s}%
_{n}\bigl(y_{1}^{n}
\bigr)\,dy_{1}^{n}.
\]
Since
\[
\mathfrak{R}_{n} ( A_{n,\delta\varepsilon_{n}} ) \leq (1+\delta
\varepsilon_{n})\mathfrak{S}_{n} ( A_{n,\delta\varepsilon_{n}} ),
\]
it follows that
\[
\lim_{n\rightarrow\infty}\mathfrak{S}_{n} ( A_{n,\delta\varepsilon_{n}} ) =1,
\]
which proves the claim.
\end{pf}

As a direct by-product of Theorem \ref{Thm:Approx_local_cond_density} and
Lemma \ref{Lemma:commute_from_p_n_to_g_n} we obtain:

\begin{theorem}
\label{ThmdeSimulation}Assume  \textup{(\ref{k/nTENDSTO1})} and \textup{(\ref{pas trop
vite}%
)} together with \textup{(\ref{E1})} and \textup{(\ref{E2})}.
Let $Y_{1}^{k}$ be a sample with
density $g_{ na }$. It holds that
\[
p_{ na } \bigl( Y_{1}^{k} \bigr)
=g_{ na }\bigl(Y_{1}^{k}\bigr)
\bigl(1+o_{G_{ na }}\bigl(\varepsilon _{n} ( \log n ) ^{2}
\bigr)\bigr).
\]
\end{theorem}

\section{\texorpdfstring{Random walks conditioned by a function of their
summands.}
{Random walks conditioned by a function of their summands}}\label{sec3}

This section extends the above results to the case when the conditioning
event is written as
%
\begin{equation}
\mathbf{U}_{1,n}:=u_{1,n} \label{cond sur f(X)}\vadjust{\goodbreak}
\end{equation}
with
\[
\mathbf{U}_{1,n}:=u ( \mathbf{X}_{1} ) + \cdots+u ( \mathbf
{X}%
_{n} ),
\]
where the function $u$ is real valued and the sequence $u_{1,n}/n$
converges. The characteristic function of the random variable $u (
\mathbf{X} ) $ is assumed to belong to $L^{r}$ for some $r\geq1$.
Let $p_{\mathbf{U}}$ denote the density of $\mathbf{U}=u ( \mathbf
{X} )$ and denote $p_{\mathbf{X}}$ the density of~$\mathbf{X}$.

Assume
%
\begin{equation}
\phi_{\mathbf{U}}(t):=E\bigl[\exp (t\mathbf{U} )\bigr]<\infty \label{PhiU}
\end{equation}
for $t$ in a nonempty neighborhood of $0$. Define the functions
$m(t)$, $%
s^{2}(t)$ and $\mu_{3}(t)$ as the first, second and third derivatives
of $%
\log\phi_{\mathbf{U}}(t)$.

Denote
%
\begin{equation}
\pi_{\mathbf{U}}^{\alpha}(u):=\frac{\exp(tu)}{\phi_{\mathbf{U}}(t)}p_{
\mathbf{U}} ( u )
\label{tiltedU}
\end{equation}
with $m(t)=\alpha$, and $\alpha$ belongs to the support of $P_{\mathbf{U}}$,
the distribution of $\mathbf{U}$.

We also introduce the family of densities
%
\begin{equation}
\pi_{u}^{\alpha}(x):=\frac{\exp(tu(x))}{\phi_{\mathbf
{U}}(t)}p_{\mathbf{X}%
} ( x )
. \label{tiltedu(X)}
\end{equation}

\subsection{\texorpdfstring{Approximation of the density of the runs.}
{Approximation of the density of the runs}}\label{sec3.1}

Assume that the sequence $\varepsilon_{n}$ satisfies (\ref{E1}) and (\ref{E2}).

Define a density $g_{u_{1,n}}(y_{1}^{k})$ on $\mathbb{R}^{k}$ as follows.
Set
\[
m_{0}:=u_{1,n}/n
\]
and
%
\begin{equation}
\label{tilte_pour_x0} g_{0}( y_{1}\rrvert
y_{0}):=\pi_{u}^{m_{0}}(y_{1})
\end{equation}
with $y_{0}$ arbitrary and, for $1\leq i\leq k-1$, define $g(
y_{i+1}\rrvert  y_{1}^{i})$ recursively. Denote $%
u_{1,i}:=u(y_{1})+ \cdots+u(y_{i})$.

Set $t_{i}$ to be the unique solution of the equation
%
\begin{equation}
m_{i}:=m(t_{i})=\frac{u_{1,n}-u_{1,i}}{n-i} \label{mif},
\end{equation}
and let
\[
s_{i}^{2}:=\frac{d^{2}}{dt^{2}} \bigl( \log E_{\pi_{\mathbf
{U}}^{m_{i}}}
\exp(t\mathbf{U}) \bigr) ( 0 )
\]
and
\[
\mu_{j}^{i}:=\frac{d^{j}}{dt^{j}} \bigl( \log
E_{\pi_{\mathbf{U}%
}^{m_{i}}}\exp(t\mathbf{U}) \bigr) ( 0 ),\qquad j=3,4,
\]
which are the second, third and fourth cumulants of $\pi_{\mathbf
{U}}^{m_{i}}$. A density $g( y_{i+1}\rrvert  y_{1}^{i})$ is
defined as
%
\begin{equation}
g\bigl( y_{i+1}\rrvert y_{1}^{i}
\bigr)=C_{i}p_{\mathbf
{X}}(y_{i+1})\mathfrak{%
n} \bigl( \alpha\beta+m_{0},\beta,u ( y_{i+1} ) \bigr).
\label{gif}
\end{equation}
Here
%
\begin{eqnarray}
\alpha&=&t_{i}+\frac{\mu_{3}^{i}}{2s_{i}^{4} ( n-i-1 ) }, \label{alpha pour f(x)}
\\
\beta&=&s_{i}^{2} ( n-i-1 ) \label{beta pour f(x)},
\end{eqnarray}
and the $C_{i}$ is a normalizing constant.

Set
%
\begin{equation}
\label{gasigmamu} g_{u_{1,n}} \bigl( y_{1}^{k} \bigr)
:=g_{0}( y_{1}\rrvert y_{0})\prod
_{i=1}^{k-1}g\bigl( y_{i+1}
\rrvert y_{1}^{i}\bigr).
\end{equation}

\begin{theorem}
\label{ThmApproxsousf(x)} Assume  \textup{(\ref{k/nTENDSTO1})} and \textup{(\ref{pas trop
vite})} together with \textup{(\ref{E1})} and \textup{(\ref{E2})}. Then \textup{(i)}
\[
p_{u_{1,n}} \bigl( Y_{1}^{k} \bigr):=p \bigl(
\mathbf {X}_{1}^{k}=Y_{1}^{k}%
\llvert \mathbf{U}_{1,n}=u_{1,n} \bigr)
=g_{u_{1,n}}\bigl(Y_{1}^{k}\bigr)
\bigl(1+o_{P_{u_{1,n}}}\bigl(\varepsilon_{n} ( \log n ) ^{2}
\bigr)\bigr)
\]
and \textup{(ii)}
\[
p_{u_{1,n}} \bigl( Y_{1}^{k} \bigr)
=g_{u_{1,n}}\bigl(Y_{1}^{k}\bigr)
\bigl(1+o_{G_{u_{1,n}}}\bigl(\varepsilon_{n} ( \log n ) ^{2}
\bigr)\bigr).
\]
\end{theorem}

\begin{pf}
We only sketch the initial step of the proof of (i), which rapidly follows
the same path as that in Theorem \ref{Thm:Approx_local_cond_density}.

As in the proof of Theorem \ref{Thm:Approx_local_cond_density}, evaluate
\begin{eqnarray*}
&& p (  \mathbf{X}_{i+1}=Y_{i+1}\rrvert
\mathbf{U}%
_{i+1,n}=u_{1,n}-U_{1,i} )
\\
&&\qquad =p_{\mathbf{X}} ( \mathbf{X}_{i+1}=Y_{i+1} )
\frac{p_{\mathbf
{U}%
} ( \mathbf{U}_{i+2,n}=u_{1,n}-U_{1,i+1} ) }{p_{\mathbf
{U}} (
\mathbf{U}_{i+1,n}=u_{1,n}-U_{1,i} ) }
\\
&&\qquad =\frac{p_{\mathbf{X}} ( \mathbf{X}_{i+1}=Y_{i+1} )
}{p_{\mathbf{U}%
} ( \mathbf{U}_{i+1}=u ( Y_{i+1} )  ) }p_{\mathbf{U}%
} \bigl( \mathbf{U}_{i+1}=u (
Y_{i+1} ) \bigr) \frac{p_{\mathbf
{U}%
} ( \mathbf{U}_{i+2,n}=u_{1,n}-U_{1,i+1} ) }{p_{\mathbf
{U}} (
\mathbf{U}_{i+1,n}=u_{1,n}-U_{1,i} ) }.
\end{eqnarray*}

Use the invariance of the conditional density with respect to the
change of
sampling defined by $\pi_{\mathbf{U}}^{m_{i}}$ to obtain
\begin{eqnarray*}
&& p (  \mathbf{X}_{i+1}=Y_{i+1}\rrvert
\mathbf{U}%
_{i+1,n}=u_{1,n}-U_{1,i} )
\\
&&\qquad =\frac{p_\mathbf{X} ( \mathbf{X}_{i+1}=Y_{i+1} ) }{p_\mathbf
{U}%
( \mathbf{U}_{i+1}=u ( Y_{i+1} )  ) }\pi_{\mathbf{U}%
}^{m_{i}} \bigl(
\mathbf{U}_{i+1}=u ( Y_{i+1} ) \bigr) \frac
{\pi_{%
\mathbf{U}}^{m_{i}} ( \mathbf{U}_{i+2,n}=u_{1,n}-U_{1,i+1} )
}{\pi
_{\mathbf{U}}^{m_{i}} ( \mathbf{U}_{i+1,n}=u_{1,n}-U_{1,i} ) }
\\
&&\qquad =p_\mathbf{X} ( \mathbf{X}_{i+1}=Y_{i+1} )
\frac{%
e^{t_{i}u(Y_{i+1})}}{\phi_{\mathbf{U}}(t_{i})}\frac{\pi_{\mathbf{U}%
}^{m_{i}} ( \mathbf{U}_{i+2,n}=u_{1,n}-U_{1,i+1} ) }{\pi
_{\mathbf{U%
}}^{m_{i}} ( \mathbf{U}_{i+1,n}=u_{1,n}-U_{1,i} ) },
\end{eqnarray*}
and proceed via the Edgeworth expansions in the above expression,
following verbatim the proof of Theorem \ref{Thm:Approx_local_cond_density}.
We omit details. The proof of (ii) follows from Lemma \ref%
{Lemma:commute_from_p_n_to_g_n}.
\end{pf}

We turn to a consequence of Theorem \ref{ThmApproxsousf(x)}.

For all $\delta>0$, let
\[
E_{k,\delta}:= \biggl\{ y_{1}^{k}\in
\mathbb{R}^{k}\dvtx \biggl\llvert \frac{%
p_{u_{1,n}} ( y_{1}^{k} ) -g_{u_{1,n}} ( y_{1}^{k} )
}{%
g_{u_{1,n}} ( y_{1}^{k} ) }\biggr\rrvert <
\delta \biggr\},
\]
which by Theorem \ref{ThmApproxsousf(x)} satisfies
%
\begin{equation}
\lim_{n\rightarrow\infty}P_{u_{1,n}} ( E_{k,\delta} ) =\lim
_{n\rightarrow\infty}G_{u_{1,n}} ( E_{k,\delta} ) =1. \label{limPu1,nGu1,n}
\end{equation}
It holds that
\begin{eqnarray*}
&&\sup_{C\in\mathcal{B} ( \mathbb{R}^{k} ) }\bigl\llvert P_{u_{1,n}} ( C\cap
E_{k,\delta} ) -G_{u_{1,n}} ( C\cap E_{k,\delta} ) \bigr\rrvert
\\
&&\qquad\leq\delta\sup_{C\in\mathcal{B} ( \mathbb{R}^{k} ) }\int_{C\cap
E_{k,\delta}}g_{u_{1,n}}
\bigl( y_{1}^{k} \bigr) \,dy_{1}^{k}\leq
\delta.
\end{eqnarray*}
By (\ref{limPu1,nGu1,n})
\[
\sup_{C\in\mathcal{B} ( \mathbb{R}^{k} ) }\bigl\llvert P_{u_{1,n}} ( C\cap
E_{k,\delta} ) -P_{u_{1,n}} ( C ) \bigr\rrvert <\eta_{n}
\]
and
\[
\sup_{C\in\mathcal{B} ( \mathbb{R}^{k} ) }\bigl\llvert G_{u_{1,n}} ( C\cap
E_{k,\delta} ) -G_{u_{1,n}} ( C ) \bigr\rrvert <\eta_{n}
\]
for some sequence $\eta_{n}\rightarrow0$; hence
\[
\sup_{C\in\mathcal{B} ( \mathbb{R}^{k} ) }\bigl\llvert P_{u_{1,n}} ( C )
-G_{u_{1,n}} ( C ) \bigr\rrvert <\delta +2\eta_{n}
\]
for all positive $\delta$. Applying Scheff\'{e}'s lemma, we have proved:

\begin{theorem}
Under the hypotheses of Theorem \ref{ThmApproxsousf(x)} the total variation
distance between $P_{u_{1,n}}$ and $G_{u_{1,n}}$ goes to $0$ as $n$
tends to
infinity, and
\[
\lim_{n\rightarrow\infty}\int\bigl\llvert p_{u_{1,n}} \bigl(
y_{1}^{k} \bigr) -g_{u_{1,n}} \bigl(
y_{1}^{k} \bigr) \bigr\rrvert \,dy_{1}^{k}=0.
\]
\end{theorem}

\begin{remark}
This result is to be compared with Theorem 1.6 in \cite
{DiaconisFreedman1988} and Theorem 2.15 in \cite{DemboZeitouni1996}
which provides a rate for this convergence for small $k$'s under some
additional conditions on the moment generating function of $\mathbf{U}$.
\end{remark}

\subsubsection{\texorpdfstring{Approximation under other sampling schemes.}
{Approximation under other sampling schemes}}\label{sec3.1.1}

In statistical applications the r.v.'s $Y_{i}$'s in Theorems \ref
{Thm:Approx_local_cond_density} and \ref{ThmApproxsousf(x)} may in
certain cases be sampled under some other distribution than $P_{ na }$ or
$G_{ na }$.

Consider the following situation.

The model consists of an exponential family $\mathcal{P}:=\{P_{\theta
,\eta
},(\theta,\eta)\in\mathcal{N}\}$ defined on~$\mathbb{R}$ with canonical
parametrization $(\theta,\eta)$ and sufficient statistics $ (
t,u ) $ defined on $\mathbb{R}$ through the densities
%
\begin{equation}
p_{\theta,\eta}(x):=\frac{dP_{\theta,\eta} ( x ) }{dx}=\exp \bigl( \theta t(x)+\eta u(x)-K(
\theta,\eta) \bigr) h(x). \label{modexp}
\end{equation}

We assume that both $\theta$ and $\eta$ belong to $\mathbb{R}$. The
natural parameter space $\mathcal{N}$ is a convex set in $\mathbb{R}^{2}$
defined as the domain of
\[
k(\theta,\eta):=\exp \bigl(K(\theta,\eta) \bigr)=\int\exp \bigl(\theta t(x)+\eta
u(x) \bigr)h(x)\,dx.
\]

For the statistician, $\theta$ is the parameter of interest whereas
$\eta$
is a nuisance one. The unknown parameter of the i.i.d. sample $\mathbf
{X}%
_{1}^{n}:= ( \mathbf{X}_{1},\ldots,\mathbf{X}_{n} ) $ observed as
$%
X_{1}^{n}:= ( X_{1},\ldots,X_{n} ) $ is  $(\theta_{T},\eta_{T})$.

Conditioning on a sufficient statistic for the nuisance parameter produces
a new exponential family which is free of $\eta$. For any $\theta$
denote $%
\widehat{\eta}_{\theta}$ the MLE of $\eta_{T}$ in model~(\ref{modexp})
parametrized in $\eta$, when $\theta$ is fixed. A classical solution for
the estimation of $\theta_{T}$ consists in maximizing the likelihood
\[
L \bigl(  \theta\rrvert X_{1}^{n} \bigr):=\prod
_{i=1}^{n}p_{\theta,\widehat{\eta}_{\theta}}(X_{i})
\]
with respect to $\theta$. This approach produces satisfactory results
when $\widehat{%
\eta}_{\theta}$ is a consistent estimator of $\eta_{\theta}$. However
for curved exponential families, it may happen that for some $\theta$ the
likelihood
\[
L_{\theta} \bigl(  \eta\rrvert X_{1}^{n}
\bigr):=\prod_{i=1}^{n}p_{\theta,\eta}(X_{i})
\]
is multimodal with respect to $\eta$ which may produce misestimation
in $%
\widehat{\eta}_{\theta}$, leading in turn to inconsistency in the resulting
estimates of $\theta_{T}$; see \cite{Sundberg2009}.

Consider $g_{u_{1,n}, ( \theta,\eta ) }$ defined in (\ref
{gasigmamu}) for fixed $ ( \theta,\eta ) $, with
$u_{1,n}:=u(X_{1})+\cdots+u(X_{n})$. Since $u_{1,n}$ is sufficient for
$\eta$, $p_{u_{1,n}, ( \theta,\eta ) }$ is independent of
$\eta$ for all $k$. Assume at present that the density
$g_{u_{1,n}, ( \theta,\eta ) }$ on $\mathbb{R}^{k}$
approximates $p_{u_{1,n}, ( \theta,\eta ) }$ on the sample
$X_{1}^{n}$ generated under $ ( \theta_{T},\eta_{T} )$; it
follows then that inserting any value $\eta_{0}$ in~(\ref{gasigmamu})
does not change the value of the resulting likelihood
\[
L_{\eta_{0}} \bigl(  \theta\rrvert X_{1}^{k}
\bigr):=g_{u_{1,n}, ( \theta,\eta_{0} ) }(X_{i}).
\]
Optimizing $L_{\eta_{0}} (  \theta\rrvert
X_{1}^{k} ) $
with respect to $\theta$ produces a consistent estimator of $\theta
_{T}$. We refer to
\cite{BroniatowskiCaron2012Exp} for examples and discussion.

Let $\mathbf{Y}_{1}^{n}$ be i.i.d. copies of $\mathbf{Z}$ with
distribution $Q$ and density $q$; assume that $Q$ satisfies the Cram\'er
condition $\int
( \exp(tx) ) q(x)\,dx<\infty$ for $t$ in a nonempty
neighborhood of $0$. Let $\mathbf{V}_{1,n}:=u ( \mathbf{Y}_{1}
) +\cdots+u (
\mathbf{Y}_{n} )$, and define
\[
q_{u_{1,n}} \bigl( y_{1}^{k} \bigr):=q \bigl(
 \mathbf{Y}%
_{1}^{k}=y_{1}^{k}
\rrvert \mathbf{V}_{1,n}=u_{1,n} \bigr)
\]
with distribution $Q_{u_{1,n}}$. The following theorem then holds:

\begin{theorem}
\label{ThmApproxSousAutreTheta} Assume  \textup{(\ref{k/nTENDSTO1})} and \textup{(\ref{pas
trop vite})} together with \textup{(\ref{E1})} and \textup{(\ref{E2})}. Then, with the same
hypotheses and notation as in Theorem \ref{ThmApproxsousf(x)},
\[
p \bigl( \mathbf{X}_{1}^{k}=Y_{1}^{k}
\llvert \mathbf{U}_{1,n}=u_{1,n}%
 \bigr)
=g_{u_{1,n}}\bigl(Y_{1}^{k}\bigr)
\bigl(1+o_{Q_{u_{1,n}}}\bigl(\varepsilon _{n} ( \log n ) ^{2}
\bigr)\bigr).
\]

Also the total variation distance between $Q_{u_{1,n}}$ and $P_{u_{1,n}}$
goes to $0$ as $n$ tends to infinity.
\end{theorem}

\begin{pf}
It is enough to check that Lemmas \ref{LemmaMomentsunderE_n}, \ref%
{LemmaMaxm_in} and \ref{Lemma_max_X_i_under_conditioning} hold when
$\mathbf{%
Y}$ satisfies the Cram\'er condition.\vadjust{\goodbreak}
\end{pf}

\begin{remark}
In the previous discussion $Q=P_{\theta_{T,}\eta_{T}}$ and $\mathbf{X}
_{1}^{n}$ are independent copies of $\mathbf{X}$ with distribution $%
P_{\theta,\eta_{0}}$.
\end{remark}

\subsection{For how long is the approximation valid?}\label{sec3.2}
\label{sec_how_far}

This section provides a rule leading to an effective choice of the crucial
parameter $k$ in order to achieve a given accuracy bound for the relative
error in Theorem \ref{ThmApproxsousf(x)}(ii). The accuracy of the
approximation is measured through
%
\begin{equation}
\operatorname{ERE}(k):=E_{G_{u_{1,n}}}1_{D_{k}} \bigl( Y_{1}^{k}
\bigr) \frac{%
p_{u_{1,n}} ( Y_{1}^{k} ) -g_{u_{1,n}} ( Y_{1}^{k} )
}{%
p_{u_{1,n}} ( Y_{1}^{k} ) } \label{ERE}
\end{equation}
and
%
\begin{equation}
\operatorname{VRE}(k):=\operatorname{Var}_{G_{u_{1,n}}}1_{D_{k}} \bigl( Y_{1}^{k}
\bigr) \frac{%
p_{u_{1,n}} ( Y_{1}^{k} ) -g_{u_{1,n}} ( Y_{1}^{k} )
}{%
p_{u_{1,n}} ( Y_{1}^{k} ) } \label{VRE}
\end{equation}
respectively, the expectation and the variance of the relative error of the
approximating scheme when evaluated on
\[
D_{k}:= \bigl\{ y_{1}^{k}\in\mathbb{R}^{k}
\mbox{ such that }\bigl\llvert g_{u_{1,n}}\bigl(y_{1}^{k}
\bigr)/p_{u_{1,n}} \bigl( y_{1}^{k} \bigr) -1\bigr
\rrvert <\delta_{n} \bigr\}
\]
with $\varepsilon_{n} ( \log n ) ^{2}/\delta_{n}\rightarrow0$
and $%
\delta_{n}\rightarrow0$; therefore $G_{u_{1,n}} ( D_{k} )
\rightarrow1$. The r.v.'s $Y_{1}^{k}$ are sampled under $%
g_{u_{1,n}}$. Note that the density $p_{u_{1,n}}$ is usually unknown. The
argument is somehow heuristic and informal; nevertheless the rule is simple
to implement and provides good results. We assume that the set $D_{k}$ can
be substituted by $\mathbb{R}^{k}$ in the above formulas, therefore assuming
that the relative error has bounded variance, which would require quite a
lot of work to be proved under appropriate conditions, but which seems to
hold, at least in all cases considered by the authors. We keep the above
notation omitting therefore any reference to $D_{k}$.

Consider a two-sigma confidence bound for the relative accuracy for a
given $%
k$, defining
\[
\operatorname{CI}(k):= \bigl[ \operatorname{ERE}(k)-2
\sqrt{\operatorname{VRE}(k)},
\operatorname{ERE}(k)+2\sqrt{\operatorname{VRE}(k)} \bigr].
\]

Let $\delta$ denote an acceptance level for the relative accuracy.
Accept $%
k $ until $\delta$ belongs to $\operatorname{CI}(k)$. For such $k$, the relative accuracy
is certified up to the level $5\%$ roughly.

The calculation of $\operatorname{VRE}(k)$ and $\operatorname{ERE}(k)$ should be carried out as follows.

Write
\begin{eqnarray*}
\operatorname{VRE}(k)^{2}& =&E_{P_{\mathbf{X}}} \biggl( \frac{g_{u_{1,n}}^{3} (
Y_{1}^{k} ) }{p_{u_{1,n}} ( Y_{1}^{k} ) ^{2}p_{\mathbf{X}%
} ( Y_{1}^{k} ) } \biggr)
\\
&&{} -E_{P_{\mathbf{X}}} \biggl( \frac{g_{u_{1,n}}^{2} ( Y_{1}^{k}
)} {%
p_{u_{1,n}} (Y_{1}^{k} )p_{\mathbf{X}} (Y_{1}^{k} )}%
\biggr)^{2}
\\
& =:&A-B^{2}.
\end{eqnarray*}

By the Bayes formula,
%
\begin{equation}
p_{u_{1,n}} \bigl( Y_{1}^{k} \bigr)
=p_{\mathbf{X}} \bigl( Y_{1}^{k} \bigr)
\frac{np ( \mathbf{U}_{k+1,n}/(n-k)=m(t_{k}) ) }{ (
n-k )
p ( \mathbf{U}_{1,n}/n=u_{1,n}/n ) }. \label{Jensen_dans_k_limite}
\end{equation}
The following lemma holds; see \cite{Jensen1995} and \cite{Richter1957}.

\begin{lemma}
\label{LemmaJensen} Let $\mathbf{U}_{1},\ldots,\mathbf{U}_{n}$ be i.i.d. random
variables with common density $p_{\mathbf{U}}$ on $\mathbb{R}$ and
satisfying the Cram\'er conditions with m.g.f. $\phi_{\mathbf{U}}$. Then with
$m(t)=u$,
\[
p ( \mathbf{U}_{1,n}/n=u ) =\frac{\sqrt{n}\phi_{\mathbf{U}%
}^{n}(t)\exp(-ntu)}{s(t)\sqrt{2\pi}} \bigl( 1+o(1) \bigr)
\]
when $|u|$ is bounded.
\end{lemma}

Introduce
\[
D:= \biggl[ \frac{\pi_{\mathbf{U}}^{m_{0}}(m_{0})}{p_{\mathbf{U}}(m_{0})} 
\biggr] ^{n}
\]
and
\[
N:= \biggl[ \frac{\pi_{\mathbf{U}}^{m_{k}} ( m_{k} )
}{p_{\mathbf{U}%
} ( m_{k} ) } \biggr] ^{ ( n-k ) }
\]
with $m_{k}$ defined in (\ref{mif}) and $m_{0}=u_{1,n}/n$. Define $t$
by $%
m(t)=m_{0}$. By (\ref{Jensen_dans_k_limite}) and Lemma \ref
{LemmaJensen} it
holds that
\[
p_{u_{1,n}} \bigl( Y_{1}^{k} \bigr) =\sqrt{
\frac{n}{n-k}}p_{\mathbf
{X}} \bigl( Y_{1}^{k} \bigr)
\frac{D}{N}\frac{s(t)}{s(t_{k})} \bigl( 1+o_{P_{u_{1,n}}}(1) \bigr).
\]
The approximation of $A$ is obtained through Monte Carlo simulation. Define
%
\begin{equation}
A \bigl( Y_{1}^{k} \bigr):=\frac{n-k}{n} \biggl(
\frac{g_{u_{1,n}} (
Y_{1}^{k} ) }{p_{\mathbf{X}} ( Y_{1}^{k} ) } \biggr) ^{3} \biggl( \frac{N}{D} \biggr)
^{2}\frac{s^{2}(t_{k})}{s^{2}(t)} \label{A(l)},
\end{equation}
and simulate $L$ i.i.d. samples $Y_{1}^{k}(l)$, each one made of $k$ i.i.d.
replicates under $p_{\mathbf{X}}$. Set
\[
\widehat{A}:=\frac{1}{L}\sum_{l=1}^{L}A
\bigl( Y_{1}^{k}(l) \bigr).
\]
We use the same approximation for $B$. Define
%
\begin{equation}
B \bigl( Y_{1}^{k} \bigr):=\sqrt{\frac{n-k}{n}}
\biggl( \frac{%
g_{u_{1,n}} ( Y_{1}^{k} ) }{p_{\mathbf{X}} (
Y_{1}^{k} ) }%
\biggr) ^{2} \biggl(
\frac{N}{D} \biggr) \frac{s(t_{k})}{s(t)} \label{B(l)}
\end{equation}
and
\[
\widehat{B}:=\frac{1}{L}\sum_{l=1}^{L}B
\bigl( Y_{1}^{k}(l) \bigr)
\]
with the same $Y^k_l(l)$'s as above.

Set
%
\begin{equation}
\label{VREbar} \overline{\operatorname{VRE}}(k):=\widehat{A}- ( \widehat{B} ) ^{2},
\end{equation}
which is a suitable approximation of $\operatorname{VRE}(k)$.

The curve $k\rightarrow\overline{\operatorname{ERE}}(k)$ is a proxy for (\ref{ERE})
and is
obtained through
\[
\overline{\operatorname{ERE}}(k):=1-\widehat{B}.
\]
A proxy of $\operatorname{CI}(k)$ can now be defined as
%
\begin{equation}
\label{CIbarre} \overline{\operatorname{CI}}(k):= \bigl[ \overline{\operatorname{ERE}}(k)-2\sqrt{
\overline{\operatorname{VRE}}(k)}, \overline{\operatorname{ERE}}(k)+2\sqrt{\overline{\operatorname{VRE}}(k)} \bigr].
\end{equation}

We now check the validity of the above approximation, comparing $%
\overline{\operatorname{CI}}(k)$ with $\operatorname{CI}(k)$ on a toy case.

Consider $u(x)=x$. The case when $p_{\mathbf{X}}$ is a centered exponential
distribution with variance $1$ allows for an explicit evaluation of $\operatorname{CI}(k)$
making no use of Lemma~\ref{LemmaJensen}. The conditional density $p_{ na }$
is calculated analytically, the density $g_{ na }$ is obtained through
(\ref%
{g_a}), hence providing a benchmark for our proposal. The terms
$\widehat{A}$
and $\widehat{B}$ are obtained by Monte Carlo simulation following the
algorithm presented below. Figures~\ref{table:choix_de_k_GD_Approx_100},
\ref{table:choix_de_k_GD_Vrai_100} and \ref
{table:choix_de_k_GD_Approx_1000}%
, \ref{table:choix_de_k_GD_Vrai_1000} show the increase in $\delta$
w.r.t. $%
k$ in the large deviation range, with $a$ such that $P (\mathbf
{S}_{1,n}> na  ) \simeq10^{-8}$. We have considered two cases, when
$n=100$ and when $n=1000$. These figures show that the approximation
scheme is quite accurate, since the relative error is fairly small.
Also they show that $\overline{\operatorname{ERE}}$ and $\overline{\operatorname{CI}}$ provide good
tools for the assessing the value of $k$.

Algorithms \ref{algo_g_s} and \ref{algo_calcul_k} produce the curve $%
k\rightarrow\overline{\operatorname{CI}}(k)$. The resulting $k=k_{\delta}$ is the longest
run length for which $g_{u_{1,n}}$ a good proxy for $p_{u_{1,n}}$.

\begin{algorithm}[t]
\caption{Evaluation of $g_{u_{1,n}}(y_{1}^{k})$.}
\label{algo_g_s}
\SetKwInOut{Input}{Input}\SetKwInOut{Output}{Output}
\SetKwInOut{Initialization}{Initialization}\SetKwInOut{Procedure}{Procedure}
\SetKwInOut{Return}{Return}

\Input{$y_{1}^{k}$, $p_{\mathbf{X}}$, $n$, $u_{1,n}$}
\Output{$g_{u_{1,n}} ( y_{1}^{k} )$}
\BlankLine
\Initialization{\\
$t_{0}\leftarrow m^{-1} (m_{0} ) $;\\
$g_{0}( y_{1}\rrvert  y_{0})\leftarrow$ (\ref{tilte_pour_x0});
}
\Procedure{\\
\For{$i\leftarrow1$ \KwTo$k-1$}{
$m_{i} \leftarrow$ (\ref{mif});\\
$t_{i} \leftarrow m^{-1}(m_{i})$ $*$;\\
$\alpha\leftarrow$ (\ref{alpha pour f(x)});\\
$\beta\leftarrow$ (\ref{beta pour f(x)});\\
Calculate $C_{i}$;\\
$g( y_{i+1}\rrvert  y_{1}^{i})\leftarrow$ (\ref{gif});
}
Compute $g_{u_{1,n}} ( y_{1}^{k} ) \leftarrow$ (\ref{gasigmamu});
}
\Return{$g_{u_{1,n}}(y_{1}^{k})$}
\end{algorithm}

The calculation of $g_{u_{1,n}} ( y_{1}^{k} ) $ above requires the
value of
\[
C_{i}= \biggl( \int p_{\mathbf{X}}(x)\mathfrak{n} \bigl( \alpha
\beta +m_{0},\beta,u(x) \bigr) \,dx \biggr) ^{-1}.
\]
This can be done through Monte Carlo simulation.

\begin{algorithm}[t]
\caption{Calculation of $k_{\delta}$.}
\label{algo_calcul_k}
\SetKwInOut{Input}{Input}\SetKwInOut{Output}{Output}
\SetKwInOut{Initialization}{Initialization}\SetKwInOut{Procedure}{Procedure}
\SetKwInOut{Return}{Return}

\Input{ $p_{\mathbf{X}}$, $\delta$, $n$, $u_{1,n}$, $L$}
\Output{$k_{\delta}$}
\BlankLine
\Initialization{$k=1$}
\Procedure{\\
\While{$\delta\notin\overline{\operatorname{CI}}(k)$}{
\For{$l\leftarrow1$ \KwTo$L$}{
Simulate $Y_{1}^{k}(l)$ i.i.d. with density $p_{\mathbf{X}}$;\\
$A ( Y_{1}^{k}(l) ):={}$(\ref{A(l)}) using Algorithm \ref{algo_g_s};\\
$B ( Y_{1}^{k}(l) ):={}$(\ref{B(l)}) using Algorithm \ref{algo_g_s};
}
Calculate $\overline{\operatorname{CI}}(k)\leftarrow{}$(\ref{CIbarre});

$k:=k+1$;
}
}
\Return{$k_{\delta}:=k$}
\end{algorithm}
\begin{algorithm}[b]
\caption{Simulation of $Y$ with density proportional to $p(x)\mathfrak
{n} ( \mu,\sigma^{2},x ) $.}
\label{algo_simu_Y}
\SetKwInOut{Input}{Input}\SetKwInOut{Output}{Output}
\SetKwInOut{Initialization}{Initialization}\SetKwInOut{Procedure}{Procedure}
\SetKwInOut{Return}{Return}
\Input{$p$, $\mu$, $\sigma^{2}$}
\Output{$Y$}
\Initialization{\\
Select a density $f$ on $ [ 0,1 ] $ and a positive constant
$K$ 
such that $p ( \mathfrak{N}^{-1}(x) ) \leq Kf(x)$ for all $x$
in $ [ 0,1 ] $
}
\Procedure{ \While{$Z<$ $p ( \mathfrak{N}^{-1}(X) ) $}{

Simulate $X$ with density $f$;

Simulate $U$ uniform on $ [ 0,1 ] $ independent of $X$;

Compute $Z:=KUf(X)$;
}
}
\Return{$Y:=\mathfrak{N}^{-1}(X)$}
\end{algorithm}

\begin{remark}
Solving $t_{i}=m^{-1}(m_{i})$ might be difficult. It may happen that the
inverse function of $m$ is at hand, but even when $p_{\mathbf{X}}$ is the
Weibull density and $u(x)=x$, this is not the case. We can replace step
$%
\ast$ by
\[
t_{i+1}:=t_{i}-\frac{ ( m ( t_{i} ) +u_{i} ) }{ (
n-i ) s^{2} ( t_{i} ) }.
\]
Indeed since
\[
m(t_{i+1})-m(t_{i})=-\frac{1}{n-i} \bigl(
m(t_{i})+u_{i} \bigr)
\]
use a first order approximation to derive that $t_{i+1}$ can be substituted
by $\tau_{i+1}$ defined as
\[
\tau_{i+1}:=t_{i}-\frac{1}{ ( n-i ) s^{2}(t_{i})} \bigl(
m(t_{i})+u_{i} \bigr).
\]
When $\lim_{n\rightarrow\infty}u_{1,n}/n=Eu ( \mathbf{X} )
$, the
values of the function $s^{2}(\cdot)$ are close to $\operatorname{Var}[u ( \mathbf
{X}%
) ]$, and the above approximation is appropriate. For the large deviation
case, the same argument applies, since $s^{2}(t_{i})$ keeps close to $%
s^{2}(t^{a})$.
\end{remark}

\begin{algorithm}[t]
\SetKwInOut{Input}{Input}\SetKwInOut{Output}{Output}
\SetKwInOut{Initialization}{Initialization}\SetKwInOut{Procedure}{Procedure}
\SetKwInOut{Return}{Return}

\Input{$p_{\mathbf{X}}$, $\delta$, $n$, $u_{1,n}$}
\Output{$Y_{1}^{k}$}
\Initialization{\\
Set $k\leftarrow k_{\delta}$ with Algorithm \ref{algo_calcul_k};

$t_{0}\leftarrow =m^{-1}(m_{0})$;
}
\Procedure{\\
Simulate $Y_{1}$ with density (\ref{tilte_pour_x0});

$u_{1,1}\leftarrow u(Y_{1})$;

\For{$i\leftarrow 1$ \KwTo $k-1$}{

$m_{i}\leftarrow{}$(\ref{mif});

$t_{i}\leftarrow m^{-1}(m_{i})$;

$\alpha\leftarrow{}$(\ref{alpha pour f(x)});

$\beta\leftarrow{}$(\ref{beta pour f(x)});

Simulate $Y_{i+1}$ with density $g( y_{i+1}\vert y_{1}^{i})$
using Algorithm \ref{algo_simu_Y};

$u_{1,i+1}\leftarrow u_{1,i}+u(Y_{i+1})$;
}
}
\Return{$Y_{1}^{k}$}
\caption{Simulation of a sample $Y_{1}^{k}$ with density
$g_{u_{1,n}}$.}\label{algo_simu_Y_1_k}
\end{algorithm}

\subsubsection{\texorpdfstring{Simulation of typical paths of a random walk under a conditioning
point.}{Simulation of typical paths of a random walk under a conditioning
point}}\label{sec3.2.1}

By Theorem \ref{ThmApproxsousf(x)}(ii), $g_{u_{1,n}}$ and the density
of $%
p_{u_{1,n}}$ approach each other on a family of subsets of $\mathbb{R}^{k}$
which contain the typical paths of the random walk under the conditional
density with probability going to $1$ as $n$ increases. By Lemma \ref%
{Lemma:commute_from_p_n_to_g_n} large sets under $P_{u_{1,n}}$ are also
large sets under $G_{u_{1,n}}$. It follows that long runs of typical paths
under $p_{u_{1,n}}$ can be simulated as typical paths under $g_{u_{1,n}}$
defined in (\ref{gasigmamu}) at least for large $n$.

The simulation of a sample $X_{1}^{k}$ with $g_{u_{1,n}}$ can be fast
and easy when $\lim_{n\rightarrow\infty}u_{1,n}/n=Eu ( \mathbf
{X} )$.
Indeed the r.v. $\mathbf{X}_{i+1}$ with density $g (
x_{i+1}|x_{1}^{i} ) $ is obtained through a standard
acceptance-rejection algorithm. The values of the parameters which
appear in the
Gaussian component of $g ( x_{i+1}|x_{1}^{i} ) $ in (\ref
{g-i}) are
easily calculated, and the dominating density can be chosen for all $i$
as $%
p_{\mathbf{X}}$. The constant in the acceptance rejection algorithm is
then $%
1/\sqrt{2\pi\beta}$. This is in contrast with the case when the
conditioning value is in the range of a large deviation event, that is,
$%
\lim_{n\rightarrow\infty}u_{1,n}/n\neq Eu ( \mathbf{X} ) $, which
appears in a natural way in importance sampling estimation for rare event
probabilities; then MCMC techniques can be used.

\begin{figure}

\includegraphics{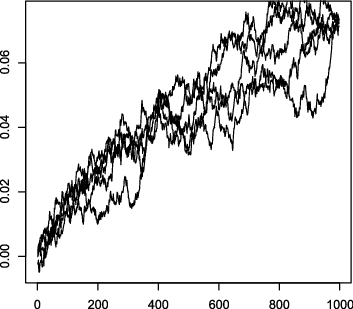}

\caption{Trajectories in the normal case for $P_{n}=10^{-2}$.}\label
{table:traj_X_Norm_MD}
\end{figure}

\begin{figure}[b]

\includegraphics{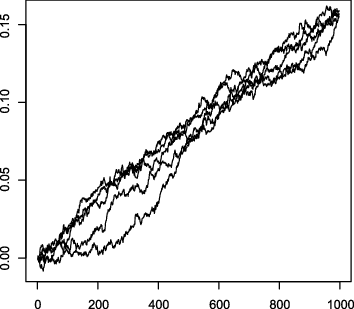}

\caption{Trajectories in the normal case for $P_{n}=10^{-8}$.}\label
{table:traj_X_Norm_GD}
\end{figure}

Denote $\mathfrak{N}$ the c.d.f. of a normal variate with parameter
$ (
\mu,\sigma^{2} ) $ and $\mathfrak{N}^{-1}$ its inverse.

%

\begin{remark}
Simulation of $Y_{1}$ can be performed through the method suggested in
\cite%
{BarbeBroniatowski1999}.
\end{remark}

Figures~\ref{table:traj_X_Norm_MD}, \ref{table:traj_X_Norm_GD}, \ref%
{table:traj_X_Exp_MD} and \ref{table:traj_X_Exp_GD} present a number of
simulations of random walks conditioned on their sum with $n=1000$ when
$%
u(x)=x$. In the Gaussian case, when the approximating scheme is known
to be
optimal up to $k=n-1$, the simulation is performed with $k=999$ and two
cases are considered: the moderate deviation case is assumed to be modeled
when $P(\mathbf{S}_{1,n}> na )=10^{-2}$ (Figure~\ref{table:traj_X_Norm_MD});
that this range of probability is in the ``moderate deviation'' range is a
commonly assessed statement among statisticians. The large deviation case
pertains to $P(\mathbf{S}_{1,n}> na )=10^{-8}$ (Figure~\ref%
{table:traj_X_Norm_GD}). The centered exponential case with $n=1000$
and $%
k=800$ is presented in Figures~\ref{table:traj_X_Exp_MD} and \ref%
{table:traj_X_Exp_GD}, under the same events.

%
%

\begin{figure}

\includegraphics{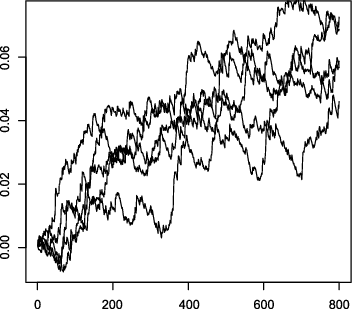}

\caption{Trajectories in the exponential case for $P_{n}=10^{-2}$.}\label
{table:traj_X_Exp_MD}
\end{figure}

\begin{figure}[b]

\includegraphics{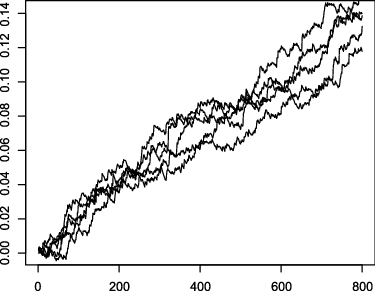}

\caption{Trajectories in the exponential case for $P_{n}=10^{-8}$.}\label
{table:traj_X_Exp_GD}
\end{figure}

In order to check the accuracy of the approximation, Figures \ref%
{table:hist_X_Norm_MD}, \ref{table:hist_X_Norm_GD} (normal case, $n=1000$,
$k=999$) and Figures~\ref{table:hist_X_Exp_MD}, \ref{table:hist_X_Exp_GD}
(centered exponential case, $n=1000$, $k=800$) present the histograms of the
simulated $\mathbf{X}_{i}$'s together with the tilted
densities at
point $a$ which are known to be the limit density of $\mathbf{X}_{1}$
conditioned on $\mathcal{E}_{n}$ in the large deviation case, and to be
equivalent to the same density in the moderate deviation case, as can be
deduced from \cite{ERmakov2006}. The tilted density in the Gaussian
case is
the normal with mean $a$ and variance $1$; in the centered exponential case
the tilted density is an exponential density on $ ( -1,\infty
) $
with parameter $1/(1+a)$.

\begin{figure}

\includegraphics{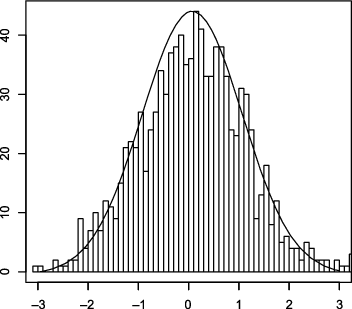}

\caption{Histogram of the $\mathbf{X}_{i}$'s in the normal case
with $n=1000$ and $k=999$ for $P_{n}=10^{-2}$. The curve represents the
associated tilted density.}
\label{table:hist_X_Norm_MD}
\end{figure}

\begin{figure}[b]

\includegraphics{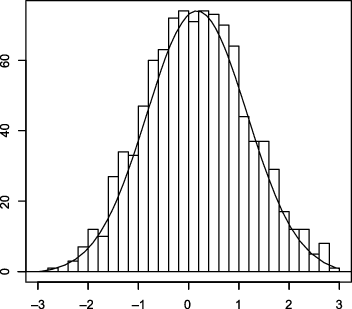}

\caption{Histogram of the $\mathbf{X}_{i}$'s in the normal case
with $n=1000$ and $k=999$ for $P_{n}=10^{-8}$. The curve represents the
associated tilted density.}
\label{table:hist_X_Norm_GD}
\end{figure}

\begin{figure}

\includegraphics{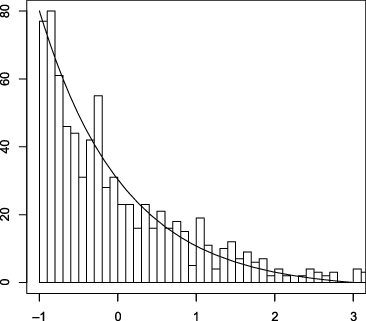}

\caption{Histogram of the $\mathbf{X}_{i}$'s in the exponential
case with $n=1000$ and $k=800$ for $P_{n}=10^{-2}$. The curve
represents the
associated tilted density.}
\label{table:hist_X_Exp_MD}
\end{figure}

\begin{figure}[b]

\includegraphics{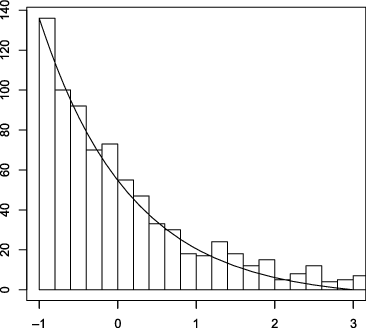}

\caption{Histogram of the $\mathbf{X}_{i}$'s in the exponential
case with $n=1000$ and $k=800$ for $P_{n}=10^{-8}$. The curve
represents the
associated tilted density.}
\label{table:hist_X_Exp_GD}
\end{figure}

Consider now the case when $u(x)=x^{2}$. Figure~\ref{table:u_carr}
presents the case when $\mathbf{X}$ is $N(0,1)$, $n=1000,k=800$, $P (
\mathbf{U}_{1,n}=u_{1,n} ) \simeq10^{-2}$. We present the histograms
of the $X_{i}$'s together with the graph of the corresponding tilted
density; when $\mathbf{X}$ is $N(0,1)$, then $\mathbf{X}^{2}$ is $\chi^{2}$.
It is well known that when $u_{1,n}/n$ is fixed to be larger than~$1$,
then the
limit distribution of $\mathbf{X}_{1}$ conditioned on $ ( \mathbf{U}
_{1,n}=u_{1,n} ) $ tends to $N ( 0,a ) $ which is the
Kullback--Leibler projection of $N(0,1)$ on the set of all probability
measures $Q$ on $\mathbb{R}$ with $\int x^{2}\,dQ(x)=a:=\lim_{n\rightarrow
\infty}u_{1,n}/n$. This distribution is precisely $g_{0}(
y_{1}\rrvert  y_{0})$ defined above. Also consider (\ref{gif});
the expansion using the definitions (\ref{alpha pour f(x)}) and (\ref
{beta pour
f(x)}) prove that as $n\rightarrow\infty$ the dominating term in $%
g_{i}( y_{i+1}\rrvert  y_{1}^{i})$ is precisely $N (
0,m_{0} ) $, and the terms including $y_{i+1}^{4}$ in the exponential
stemming from $\mathfrak{n} ( \alpha\beta+m_{0},\beta
,u(y_{i+1}) ) $ are of order $O ( 1/ ( n-i )  ) $;
the terms depending on $y_{1}^{i}$ are of smaller order. The fit which is
observed in Figure~\ref{table:u_carr} is in accordance with the above
statement in the LDP range (when $\lim_{n\rightarrow\infty} u_{1,n}/n\neq1$), and with the MDP approximation when $\lim_{n\rightarrow
\infty} u_{1,n}/n=1$ and $\lim\inf_{n\rightarrow\infty} (
u_{1,n}-n ) /\sqrt{n}\neq0$, following \cite{ERmakov2006}.

\begin{figure}

\includegraphics{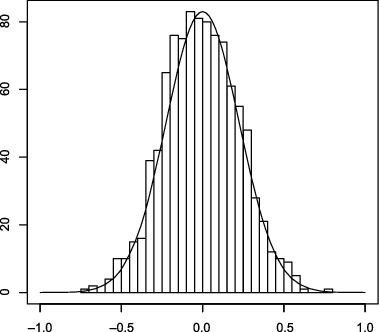}

\caption{Histogram of the $\mathbf{X}_{i}$'s in the normal case
with $n=1000$, $k=800$ and $u(x)=x^{2}$ for $P_{n}=10^{-2}$. The curve
represents the associated tilted density.}
\label{table:u_carr}
\end{figure}

\section{\texorpdfstring{Conditioning on large sets.}
{Conditioning on large sets}}\label{sec4}
\label{sec_IS}

The approximation of the density
\[
p_{A_{n}} \bigl( \mathbf{X}_{1}^{k}=Y_{1}^{k}
\bigr):=p \bigl(  \mathbf{%
X}_{1}^{k}=Y_{1}^{k}
\rrvert \mathbf{U}_{1,n}\in A_{n} \bigr)
\]
of the runs $\mathbf{X}_{1}^{k}$ under large sets $ ( \mathbf{U}%
_{1,n}\in A_{n} ) $ for Borel sets $A_{n}$ with nonempty interior
follows from the above results through integration. Here, in the same vein
as previously, $Y_{1}^{k}$ is generated under $P_{A_{n}}$. An
application of
this result for the evaluation of rare event probabilities through
importance sampling is briefly presented in the next section. The present
section pertains to the large deviation case.

\subsection{\texorpdfstring{Conditioning on a large set
defined through the density of its
dominating point.}
{Conditioning on a large set defined through the density of its
dominating point}}\label{sec4.1}

We focus on cases when $ ( \mathbf{U}_{1,n}\in A_{n} ) $ can
be expressed as $ ( \mathbf{U}_{1,n}/n\in A ) $ where $A$ is a
fixed Borel set (independent of $n$) with essential infimum $\alpha$
larger than $E\mathbf{U}$ and which can be described as a ``thin'' or
``thick'' Borel set according to its local density at point $\alpha$.

The starting point is the approximation of $p_{nv}$ on $\mathbb{R}^{k}$ for
large values of $k$ under the conditioning point
\[
\mathbf{U}_{1,n}/n=v
\]
when $v$ belongs to $A$. Denote $g_{nv}$ the corresponding approximation
defined in (\ref{gasigmamu}). It holds that
%
\begin{equation}
p_{nA}\bigl(x_{1}^{k}\bigr)=\int
_{A}p_{nv} \bigl( \mathbf {X}_{1}^{k}=x_{1}^{k}
\bigr) p( \mathbf{U}_{1,n}/n=v\rrvert \mathbf{U}_{1,n}
\in nA)\,ds. \label{etoile}
\end{equation}
In contrast with the classical importance sampling approach for this problem
we do not consider the dominating point approach, but merely realize a sharp
approximation of the integrand at any point of the domain $A$ and consider
the dominating contribution of all those distributions in the
evaluation of
the conditional density $p_{nA}$. A similar point of view has been
considered in \cite{BarbeBroniatowski2004} for sharp approximations of
Laplace-type integrals in $\mathbb{R}^{d}$.

Turning to (\ref{etoile}) it appears that what is needed is a sharp
approximation for
%
\begin{equation}
p( \mathbf{U}_{1,n}/n=v\vert\mathbf{U}_{1,n}\in nA)=
\frac{p(%
\mathbf{U}_{1,n}/n=v){\mathbh{1}}_{A}(v)}{P(\mathbf{U}_{1,n}\in nA)} \label{densiteLocale1}
\end{equation}
with some uniformity for $v$ in $A$. We will assume that $A$ is bounded
above in order to avoid further regularity assumptions on the distribution
of $\mathbf{U}$.

Recall that the \textit{essential infimum }\texttt{essinf}$A=\alpha$
of the
set $A$ with respect to the Lebesgue measure is defined through
\[
\alpha:=\inf \bigl\{ x\dvtx \mbox{ for all }\varepsilon>0,\bigl\llvert [ x,x+
\varepsilon ] \cap A\bigr\rrvert >0 \bigr\}
\]
with $\inf\varnothing:=-\infty$.

We assume that $\alpha>-\infty$, which is tantamount to saying that
we do not
consider very thin sets (e.g., not Cantor-type sets).

The density of the point $\alpha$ in $A$ will not be measured in the
ordinary way, through
\[
d(\alpha):=\lim_{\varepsilon\rightarrow0}\frac{\llvert  A\cap [
\alpha-\varepsilon,\alpha+\varepsilon ] \rrvert
}{\varepsilon},
\]
but through the more appropriate quantity
\[
M(t):=t\int_{A-\alpha}e^{-ty}\,dy,\qquad t>0.
\]
For any set $A$, $0\leq M(t)\leq1$. If there exists an interval $ [
\alpha,\alpha+\varepsilon ] \subset A$, then $\lim_{t\rightarrow
\infty}M(t)=1$. As an example, for a self similar set $A:=A_{p}$
defined as $A_{p}:=\bigcup_{n\in\mathbb{Z}}p^{n}I_{p}$ where $p>2$
and $I_{p}:= [  ( p-1 ) /p,1 ]$, it holds that
$0=\mathrm{essinf} A_{p}$ and $pA_{p}=A_{p}$. Consequently for any
$t\geq0$, $M(tp)=M(t)$ and $M(tp)=M(t)$ for all $t\geq0$; it follows that
\[
\inf_{1\leq u\leq p}M(u)=\lim\inf_{t\rightarrow\infty}M(t)\leq\lim
\sup_{t\rightarrow\infty}M(t)=\sup_{1\leq u\leq p}M(u).
\]

Define
\[
M_{n}(t):=M(nt)/t=\int_{A-\alpha}e^{-ty}\,dy
\]
and
\[
\Psi_{n}(t):=n\log\phi_{\mathbf{U}}(t)+\log M_{n}(t)-n
\alpha t
\]
for all $t>0$ such that $\phi_{\mathbf{U}}(t)$ is finite. We borrow from
\cite{BarbeBro2000} the following results.

Define $\mu_{n}(t):= ( 1/n ) \log M_{n}(t)$ which is for all $
n\geq1$ a decreasing function of $t$ on $ ( 0,\infty ) $, and
which is negative for large $n$. Also $\mu_{n}^{\prime}(t)=\mu
_{1}^{\prime}(nt)$ and $\mu_{1}^{\prime}$ are nondecreasing on
$ (
0,\infty ) $.

Let $\overline{\mu}:=\lim_{t\rightarrow\infty}\mu_{1}^{\prime}(t)$
and $%
\underline{\mu}:=\lim_{t\rightarrow0}\mu_{1}^{\prime}(t)$. Then
according to \cite{BarbeBro2000} the following holds:
%
\begin{lemma}
Under the above notation and hypotheses, the equation\break  $\Psi_{n}^{\prime
}(t)=0$ has a unique solution $t_{n}$ in $ ( 0,t_{0} ) $ for $%
\alpha$ in $ ( E\mathbf{U}+\overline{\mu},\infty ) $ where $
t_{0}:=\sup \{ t\dvtx \phi_{\mathbf{U}}(t)<\infty \}$. Furthermore
if $\alpha>E\mathbf{U}+\underline{\mu}$, then there exists a compact
set $K\subset
( 0,t_{0} ) $ such that $t_{n}\in K$ for all $n$.
\end{lemma}

Assume that $\alpha>E\mathbf{U}+\underline{\mu}$. Define $\psi
_{n}(t):=\Psi_{n}^{\prime\prime}(t)$, and suppose that for any
$\lambda>0$,
%
\begin{equation}
\lim_{n\rightarrow\infty}\sup_{\llvert  u\rrvert <\lambda}\frac
{%
\psi_{n} ( t_{n}+{u}/{\sqrt{\psi_{n}(t_{n})}} ) }{\psi
_{n} ( t_{n} ) }=1,
\label{selfneglecting}
\end{equation}
where $t_{n}$ is a solution of $\Psi_{n}^{\prime}(t)=0$ in the range
$ (
0,t_{0} ) $. It can be proved that (\ref{selfneglecting}) holds, for
example, when $t\rightarrow\log M(t)/t$ is a regularly varying
function at
infinity with index $\rho\in ( 0,1 )$, that is, $\log
M(t)/t\in
\mathcal{R}_{\rho} ( \infty ) $; see \cite{BarbeBro2000}, Lemma~2.2.

We also assume that
%
\begin{equation}
\lim\sup_{t\rightarrow\infty}t \bigl( \log M(t) \bigr) <\infty,
\label{condRro}
\end{equation}
which holds, for example, when $\log ( M(t)/t ) \in
\mathcal{R}%
_{\rho} ( \infty )$, for $0\leq\rho<1$.

Theorem 2.1 in \cite{BarbeBro2000} provides a general result to be inserted
in (\ref{densiteLocale1}); we take the occasion to correct a misprint in
this result.

\begin{theorem}
Assume (\ref{selfneglecting}) and (\ref{condRro}) together with the
aforementioned conditions on the r.v. $\mathbf{U}$. Then for $\alpha>E%
\mathbf{U}+\underline{\mu}$,
%
\begin{equation}
P(\mathbf{U}_{1,n}\in nA)=\frac{\phi_{\mathbf{U}%
}^{n}(t_{n})M_{n}(t_{n})e^{-nt_{n}\alpha}}{\sqrt{\psi_{n} (
t_{n} ) }\sqrt{2\pi}} \bigl( 1+o(1) \bigr)\qquad
\mbox{as }n\rightarrow \infty, \label{LDP}
\end{equation}
with $t_{n}$ satisfying $\Psi_{n}^{\prime}(t)=0$ provided that the
function $x\rightarrow P(\mathbf{U}_{1,n}\in nA+x)$ is nonincreasing
for $n$
large enough. In particular, this last condition holds if
\begin{longlist}[(iii)]
\item[(i)] (Petrov): $A= ( \alpha,\infty ) $ or $A=[\alpha,\infty
)$; in
this case $M_{n}(t)=1/t$; note that in this case the classical result is
slightly different, since
\[
P(\mathbf{U}_{1,n}> na )=\frac{\phi_{\mathbf{U}}^{n}(t^{a})e^{-nt^{a}a}}{
t^{a}s(t^{a})\sqrt{2\pi}} \bigl( 1+o(1) \bigr) \qquad\mbox{as }n\rightarrow \infty
\]
with $m(t^{a})=a$ and $a>E\mathbf{U}$; this is readily seen to be equivalent
to (\ref{LDP}) when $A= ( a,\infty ) $.

\item[(ii)] $\mathbf{U}$ has a symmetric unimodal distribution.

\item[(iii)] $\mathbf{U}$ has a strongly unimodal distribution.
\end{longlist}
\end{theorem}

The shape of $A$ near $\alpha$ is reflected in the behavior of the function
$M(t)$ for large values of $t$. As such, the larger the $n$, the more relevant
is the shape of $A$ near $\alpha$.

Note further that $M_{n}(t)e^{-nt\alpha}=\int_{A}e^{-nty}\,dy$ from
which we
see that $\alpha$ plays no role in (\ref{LDP}). Hence $\alpha$ can be
replaced by any number $\gamma$ such that $\int_{A-\gamma}e^{-ty}\,dy$
converges. Further $t_{n}$ is independent of $\alpha$. The so-called
dominating point $\alpha$ of $A$ can therefore be defined as
\[
\alpha:=\lim_{t\rightarrow\infty}\log\int_{A}e^{-ty}\,dy.
\]
In order to examine further the role played in (\ref{LDP}) by the regularity
of $A$ near its essential infimum $\alpha$, introduce the pointwise H\"
{o}%
lder dimension of $A$ at $\alpha$ as
\[
\delta(\alpha):=\frac{\log G(\varepsilon)}{-\log\varepsilon},
\]
where
\[
G(\varepsilon):=\bigl\llvert A\cap [ \alpha,\alpha+\varepsilon ]
\bigr\rrvert\qquad
\mbox{for positive }\varepsilon.
\]
We refer to Proposition 2.1 in \cite{BarbeBro2000} for a set of
Abel--Tauber-type results which link the properties of $M(t)$ at
infinity with those of $%
G $ at $0$. For example, it follows that $G(\varepsilon)\sim
\varepsilon
^{\delta(\alpha)}$ (as $\varepsilon\rightarrow0$) if and only if
$M(t)\sim
ct^{-\delta(\alpha)+1}\Gamma ( 1+\delta(\alpha) ) $ (as $%
t\rightarrow\infty$). Consequently if $M_{n}(t)\rightarrow1$ as $%
t\rightarrow\infty$, then $M(t)\sim t$ as $t\rightarrow\infty$ and $
G(\varepsilon)\sim\varepsilon$ as $\varepsilon\rightarrow0$.

Asymptotic formulas for the numerator in (\ref{densiteLocale1}) are well
known and have a long history, going back to \cite{Richter1957}. It
holds that
%
\begin{equation}
p(\mathbf{U}_{1,n}/n=v)=\frac{\sqrt{n}e^{nvt^{v}}\phi_{\mathbf
{U}}(t^{v})}{%
\sqrt{2\pi}s(t^{v})} \bigl( 1+o(1) \bigr)\qquad \mbox{as }n\rightarrow \infty \label{numLDP}
\end{equation}
with $t^{v}$ defined as $m(t^{v})=v$.

Plugging in (\ref{numLDP}) and (\ref{LDP}) in (\ref{etoile}) provides an
expression for the density of the runs. For applications the only relevant
case is developed in the following paragraph.

\subsection{\texorpdfstring{Conditioning on a thick set.}
{Conditioning on a thick set}}\label{sec4.2}

In the case when $A= ( a,\infty ) $ or with $a>Eu (
\mathbf{X}%
) $ or, more generally, when $A$ is a thick set in a neighborhood
of its
essential infimum [i.e., when $\lim_{t\rightarrow\infty}M(t)=1$] a simple
asymptotic evaluation for (\ref{densiteLocale1}) when $A$ is unbounded can
be obtained. Indeed an expansion of the ratio yields
%
\begin{equation}
p( \mathbf{U}_{1,n}/n=v\rrvert \mathbf{U}_{1,n}> na )=
\bigl( nt\exp\bigl(-nt(v-a)\bigr) \bigr) {\mathbh{1}}_{A}(v)
\bigl(1+o(1)\bigr) \label{Approx_Exp}
\end{equation}
with $m(t)=a$, indicating that $\mathbf{U}_{1,n}/n$ is roughly exponentially
distributed on $A$ with expectation $a+1/nt$. This result is used in Section~\ref{sec5} in order to derive estimators of some rare event probabilities through
importance sampling.

In order to obtain a sharp approximation for $p_{nA} ( \mathbf{X}%
_{1}^{k}=Y_{1}^{k} ) $ it is necessary to introduce an interval
$ (
a,a+c_{n} ) $ which contains the principal part of the integral~(\ref%
{etoile}).

Let $c_{n}$ denote a positive sequence such that the following
condition (C)
holds:
\begin{eqnarray*}
\lim_{n\rightarrow\infty}nc_{n} &=&\infty,
\\
\sup_{n\geq1}\frac{nc_{n}}{(n-k)} &<&\infty
\end{eqnarray*}
and denote $c$ the current term $c_{n}$.

Define on $\mathbb{R}^{k}$ the density
%
\begin{eqnarray}\label{h1_k}
&& g_{nA}\bigl(y_{1}^{k}\bigr)
\nonumber
\\[-8pt]
\\[-8pt]
\nonumber
&&\qquad:=\frac{nm^{-1} ( a ) \int_{a}^{a+c}g_{nv}(y_{1}^{k}) (
\exp(
-nm^{-1} ( a ) (v-a)) ) \,dv}{1-\exp(-nm^{-1} ( a
) c)}.
\end{eqnarray}
The density
%
\begin{equation}
\frac{nm^{-1} ( a )  ( \exp(-nm^{-1} ( a )
(v-a)) ) {\mathbh{1}}_{ ( a,a+c ) }(v)}{1-\exp
(-nm^{-1} (a ) c)}, \label{condS}
\end{equation}
which appears in (\ref{h1_k}) approximates $p( \mathbf{U}%
_{1,n}/n=v\rrvert  a<\mathbf{U}_{1,n}/n<a+c)$. Furthermore due to
Theorem %
\ref{ThmApproxsousf(x)} $g_{nv}(Y_{1}^{k})$ approximates $p_{nv}(Y_{1}^{k})$
when $Y_{1}^{k}$ results from sampling under $P_{nA}$. For a discussion on
the maximal value of $k$ for which a given relative accuracy is attained,
see \cite{BroniatowskiCaron2011IS}.

The \textit{variance function} $V$ of the distribution of $\mathbf{U}$
is defined on the span of $\mathbf{U}$ through
\[
v\rightarrow V(v):=s^{2}\bigl(m^{-1}(v)\bigr).
\]
Denote (V) the condition
\[
\sup_{n\geq1}\sqrt{n}\int_{a}^{\infty}V^{\prime}(v)
\bigl( \exp \bigl(-nm^{-1}(a) ( v-a )\bigr) \bigr) \,dv<\infty.
\]

\begin{theorem}
\label{Thm_approx_largeSets} Assume \textup{(\ref{E1}), (\ref{E2}), (C), (V)}.
Then for any positive $\delta<1$:

\textup{(i)}
%
\begin{equation}
p_{nA} \bigl( \mathbf{X}_{1}^{k}=Y_{1}^{k}
\bigr) =g_{nA}\bigl(Y_{1}^{k}\bigr)
\bigl(1+o_{P_{nA}}(\delta_{n})\bigr) \label{Thm_approx_largeSets(i)}
\end{equation}
and \textup{(ii)}
%
\begin{equation}
p_{nA} \bigl( \mathbf{X}_{1}^{k}=Y_{1}^{k}
\bigr) =g_{nA}\bigl(Y_{1}^{k}\bigr)
\bigl(1+o_{G_{nA}}(\delta_{n})\bigr), \label{Thm_approx_largeSets(ii)}
\end{equation}
where
%
\begin{equation}
\delta_{n}:=\max \bigl( \varepsilon_{n} ( \log n )
^{2}, \bigl( \exp(-nc) \bigr) ^{\delta} \bigr). \label{vitesse}
\end{equation}
\end{theorem}

\begin{pf}
See the \hyperref[app]{Appendix}.
\end{pf}

\begin{remark*}Most distributions used in statistics satisfy (V); numerous
papers have focused on the properties of variance functions and
classification of distributions; see, for example, \cite{LetacMora} and
references
therein.
\end{remark*}

\begin{corollary}
Under the hypotheses of Theorem \ref{Thm_approx_largeSets} the total
variation distance between $P_{nA}$ and $G_{nA}$ goes to $0$ as $n$
tends to
infinity, that is,
\[
\lim_{n\rightarrow\infty}\int\bigl\llvert p_{nA} \bigl(
y_{1}^{k} \bigr) -g_{nA} \bigl(
y_{1}^{k} \bigr) \bigr\rrvert \,dy_{1}^{k}=0.
\]
\end{corollary}

\section{\texorpdfstring{Applications.}{Applications}}\label{sec5}

\subsection{\texorpdfstring{Rao--Blackwellization of estimators.}
{Rao--Blackwellization of estimators}}\label{sec5.1}

This example illustrates the role of Theorem \ref{ThmApproxsousf(x)} in
statistical inference; the conditioning event is local, in the range
where $\lim_{n\rightarrow\infty}u_{1,n}/n=Eu ( \mathbf{X} ) $.

In statistics the following situation is often encountered. A model
$\mathcal{P}$ consists of a family of densities $p_{\theta}$ where the
parameter $\theta$ is assumed to belong to $\mathbb{R}^{d}$, and a
sample of i.i.d. r.v.'s $\mathbf{X}_{1}^{n}$ is observed, with each of
the $\mathbf{X}_{i}$'s having density $p_{\theta_{T}}$ where $\theta
_{T}$ is unknown; denote $X_{1},\ldots,X_{n}$ the observed data set. Let
$\mathbf{U}_{1,n}:= u ( \mathbf{X}_{1} ) +\cdots+u ( \mathbf{X}_{n} ) $ and let
$u_{1,n}:=u(X_{1})+\cdots+u(X_{n})$, which usually satisfies $\lim_{n\rightarrow\infty}u_{1,n}/n=Eu ( \mathbf{X} ) $. A
preliminary estimator $\widehat{\theta} ( \mathbf{X}_{1}^{n}
) $
is chosen, which may have the advantage of being easily computable, at the
cost of having poor efficiency, approaching $\theta_{T}$ loosely in terms
of the MSE. The famous Rao--Blackwell theorem asserts that the MSE of the
conditional expectation of $\widehat{\theta} ( \mathbf{X}%
_{1}^{n} ) $ given the observed value $u_{1,n}$ of any statistic
improves on the MSE of $\widehat{\theta} ( \mathbf{X}_{1}^{n} )$.
When $u_{1,n}$ is sufficient for $\theta$ the reduction is maximal, leading
to the unbiased minimal variance estimator for $\theta_{T}$ when
$\widehat{%
\theta} ( \mathbf{X}_{1}^{n} ) $ is unbiased
(Lehmann--Scheff\'{e}
theorem).

The conditional density $p_{u_{1,n}} ( x_{1}^{n} ):=p (
\mathbf{X}_{1}^{n}=x_{1}^{n}\rrvert \mathbf{U}_{1,n}=u_{1,n} )
$ is
usually unknown, and Rao--Blackwellization of estimators cannot be performed
in many cases. Simulations of long runs of length $k=k_{n}$ under a
proxy of
$p_{u_{1,n}} ( x_{1}^{k} ) $ provide an easy way to improve the
preliminary estimator, averaging values of $\widehat{\theta} (
(
X_{1}^{k} )(l) ) _{1\leq l\leq L}$ where the samples $ (
X_{1}^{k}(l) )$'s are obtained under the approximation of $%
p_{u_{1,n}} ( x_{1}^{k} ) $ and $L$ runs are performed$.$

Consider the Gamma density
%
\begin{equation}
f_{\rho,\theta}(x):=\frac{\theta^{-\rho}}{\Gamma(\rho)}x^{\rho
-1}\exp\biggl( -
\frac{x}{\theta}\biggr) \qquad\mbox{for }x>0. \label{gamma}
\end{equation}
As $\rho$ varies in $\mathbb{R}^{+}$ and $\theta$ is positive, the density
belongs to an exponential family $\gamma_{r,\theta}$ with parameters $
r:=\rho-1$ and $\theta$, and sufficient statistics are $t(x):=\log x$
and $%
u(x):=x$, respectively, for $r$ and $\theta$. Given an i.i.d. sample $%
X_{1}^{n}:= ( X_{1},\ldots,X_{n} ) $ with density $\gamma
_{r_{T},\theta_{T}}$ the resulting sufficient statistics are,
respectively, $%
t_{1,n}:=\log X_{1}+\cdots+\log X_{n}$ and $u_{1,n}:=X_{1}+\cdots+X_{n}$. We
consider the parametic model $ ( \gamma_{r_{T},\theta},\theta
\geq{0}%
) $ assuming $r_{T}$ known.

Definition (\ref{gasigmamu}) shows that $g_{u_{1,n}}$ depends on the unknown
parameter $\theta_{T}$. It can be seen that $u_{1,n}$ is nearly
sufficient for $\theta$ in $g_{u_{1,n}}$\vadjust{\goodbreak} in the sense that the value
of $%
g_{u_{1,n}} ( X_{1}^{k} ) $ does not vary when $\theta_{T}$ is
substituted by any other value $\theta$ of the parameter and the $X_{i}$'s
are generated under any density $\gamma_{r_{T},\theta^{\prime}}$ (see
\cite{BroniatowskiCaron2012Exp}) this is indeed in agreement with the
statement of Theorem \ref{ThmApproxSousAutreTheta}. Hence on one hand $,
u_{1,n}$ can be used to obtain improved estimators of $\theta_{T}$
and on the other hand, $g_{u_{1,n}}$ can be used to simulate
samples distributed under a proxy of $p_{u_{1,n}}$ using any $\theta$ in
lieu of $\theta_{T}$ in~(\ref{gasigmamu}), as is done in the following
procedure:

A first unbiased estimator of $\theta_{T}$ is chosen as
\[
\widehat{\theta}_{2}:=\frac{X_{1}+X_{2}}{2r_{T}}.
\]
Given an i.i.d. sample $X_{1}^{n}$ with density $\gamma_{r_{T},\theta_{T}}$
the Rao--Blackwellized estimator of $\widehat{\theta}$ is defined as
\[
\theta_{RB,2}:=E (  \widehat{\theta}_{2}\rrvert
\mathbf {U}%
_{1,n} )
\]
whose variance is less than $\operatorname{Var} \widehat{\theta}_{2}$.

Consider $k=2$ in $g_{u_{1,n}}(y_{1}^{k})$, and let $ (
Y_{1},Y_{2} ) $ be distributed according to $g_{u_{1,n}}(y_{1}^{2})$.
Replicates of $ ( Y_{1},Y_{2} ) $ induce an estimator of
$\theta
_{RB,2}$ for fixed $u_{1,n}$. Iterating on the simulation of the runs
$X_{1}^{n}$ produces for $n=100$ an i.i.d. sample of $\theta_{RB,2}$'s
from which $\operatorname{Var} \theta_{RB,2}$ is estimated. The resulting variance
shows a net
improvement with respect to the estimated variance of $\widehat{\theta}
_{2}$. It is of some interest to investigate this gain in efficiency
as the
number of terms involved in $\widehat{\theta}_{k}$ increases together
with $%
k$. As $k$ approaches $n$ the variance of $\widehat{\theta}_{k}$ approaches
the Cram\'er--Rao bound. Figure \ref{fig14} shows the decay of the variance
of $%
\widehat{\theta}_{k}$. We note that whatever the value of $k$ the estimated
value of the variance of $\theta_{RB,k}$ is constant, and is quite
close to
the
Cram\'er--Rao bound. This is indeed an illustration of Lehmann--Scheff\'
{e}%
's theorem.

\begin{figure}

\includegraphics{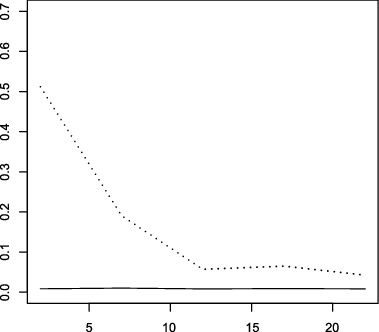}

\caption{Variance of $\widehat{\protect\theta}_{k}$, the initial estimator
(dotted line), along with the variance of $\protect\theta_{RB,k}$, the
Rao--Blackwellized estimator (solid line) with $n=100$ as a function of
$k$.}\label{fig14}
\end{figure}

\subsection{\texorpdfstring{Importance sampling for rare event probabilities.}
{Importance sampling for rare event probabilities}}\label{sec5.2}

Here we consider the application of the approximating scheme under a
conditioning event defined through a large set, where this event is
also on the
large deviation scale. A development of the present section is
presented in
\cite{BroniatowskiCaron2011IS} and in Section~3 of \cite
{Caronetal2013}; see also \cite{BroniatowskiRitov2009}. Consider the
estimation of the large deviation probability for the mean of $n$
i.i.d. r.v.'s $u(\mathbf{X}_{i})$ satisfying the conditions of this
paper. This is a benchmark problem in the study of rare events; we
refer to \cite{Bucklew2004} for the background of this section.

Let $u_{1,n}:= na $ for fixed $a$ larger than $Eu(\mathbf{X})$. The
probability to be estimated is
\[
P_{n}:=P ( \mathbf{U}_{1,n}>u_{1,n} ).
\]
The importance sampling procedure substitutes the empirical estimator
%
\begin{eqnarray}
\widehat{P}_{n} &:=&\frac{1}{L}\sum
_{l=1}^{L}\mathbh{1} \bigl( \mathbf{U} 
_{1,n}(l)>u_{1,n} \bigr)
\nonumber
\\[-8pt]
\\[-8pt]
\nonumber
&=&\frac{1}{L}\sum_{l=1}^{L}
\mathbh{1} \Biggl( \sum_{i=1}^{n}u \bigl(
\mathbf{X}%
_{i}(l) \bigr) >u_{1,n} \Biggr)
\label{IS1}
\end{eqnarray}
by
%
\begin{equation}\quad
P_{n}^{\mathrm{IS},g}:=\frac{1}{L}\sum
_{l=1}^{L}\frac{p ( u ( \mathbf{X}%
_{1}(l) )  ) \cdots p ( u ( \mathbf{X}_{n}(l) )
) }{%
g ( u ( \mathbf{X}_{1}(l) ) \cdots u ( \mathbf
{X}_{n}(l) )
) }\mathbh{1} \Biggl( \sum
_{i=1}^{n}u \bigl( \mathbf{X}_{i}(l)
\bigr) >u_{1,n} \Biggr). \label{IS2}
\end{equation}
In the above display (\ref{IS1}) the sample $\mathbf{X}_{1}^{n}(l)$ is
generated under i.i.d. sampling with distribution $P_{\mathbf{X}}$ and
the $L$ samples are i.i.d. In display (\ref{IS2}) the sample $\mathbf{X}
_{1}^{n}(l)$ is generated under the density $g$ on $\mathbb{R}^{n}$ (under
which the $\mathbf{X}_{i}$'s may not be independent). The $L$ samples $%
\mathbf{X}_{1}^{n} ( l ) $ are i.i.d.

It is well known that the optimal sampling density is
\[
p_{\mathrm{opt}} \bigl( x_{1}^{n} \bigr):=p \bigl(
 \mathbf{X}%
_{1}^{n}=x_{1}^{n}
\rrvert \mathbf{U}_{1,n}>u_{1,n} \bigr),
\]
which is not achievable since it presumes a known $P_{n}$. This optimal sampling
density produces the zero variance estimator $P_{n}$ itself with $L=1$.
However approximating $p_{\mathrm{opt}} ( x_{1}^{n} ) $ sharply at
least on
the first $k$ coordinates for large $k$ produces a large hit rate for the
importance sampling procedure, and pushes the importance factor toward 1.

Define the sampling density $g$ on $\mathbb{R}^{n}$ as
\[
g \bigl( x_{1}^{n} \bigr):=g_{nA}
\bigl(x_{1}^{k}\bigr)\prod_{i=k+1}^{n}
\pi _{u}^{a} ( x_{i} ),
\]
where $g_{nA}$ is defined in (\ref{h1_k}), and $\pi_{u}^{a}$ is the density
defined in (\ref{tiltedu(X)}). The approximating density $g_{nA}$ has been
used to simulate the $k$ first $\mathbf{X}_{i}$'s and the remaining
$n-k$'s are i.i.d. with the classical tilted density. The classical IS scheme
coincides with the present one with the difference that $k=1$ and $%
g_{A_{n}}(x_{1})=\pi_{u}^{a} ( x_{1} ) $, that is, simulating
under an
i.i.d. sampling scheme with common density $\pi_{u}^{a}$.

\begin{figure}

\includegraphics{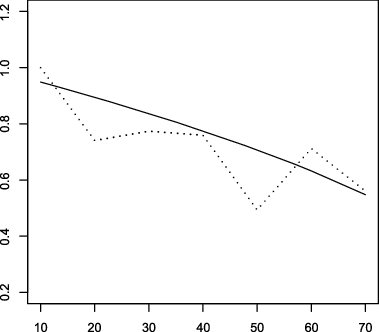}

\caption{Ratio of the empirical value of the MSE of the adaptive estimate
w.r.t. the empirical MSE of the i.i.d. twisted one (dotted line) along with
the true value of this ratio (solid line) as a function of $k$.}
\label{MSEvsk}
\end{figure}

Simulation under $g_{nA}$ is performed through a double step procedure: In
the first step, randomize the value of $\mathbf{U}_{1,n}/n$ on $ (
a,+\infty ) $ according to a proxy of its distribution
conditioned on $%
\mathbf{U}_{1,n}> na $; hence simulate a random variable $S$ on $ (
a,+\infty ) $ with density
%
\begin{equation}
p_{S}(s):=nm^{-1} ( a_{n} ) \bigl( \exp
\bigl(-nm^{-1} ( a ) (s-a)\bigr) \bigr) {\mathbh{1}}_{ ( a,+\infty ) }(s).
\label{condSglobal}
\end{equation}
Then plug in $nS$ in lieu of $u_{1,n}$ in (\ref{gasigmamu}) and iterate.
This is equivalent to considering each point in the target set as a
dominating point,
weighted by its conditional density under $ ( \mathbf
{U}_{1,n}> na  )
$. Simulation of $S$ under (\ref{condSglobal}) instead of (\ref{condS}) is
slightly suboptimal but much simpler. It can be proved that the MSE of the
estimate of $P_{n}$ in this new IS sampling scheme is reduced by a
factor $%
\sqrt{ ( n-k ) /n}$ with respect to the classical scheme when
calculated on large subsets of $\mathbb{R}^{k}$; see \cite%
{BroniatowskiCaron2011IS}. Figure~\ref{MSEvsk} shows, in a simple case, the
ratio of the empirical value of the MSE of the adaptive estimate w.r.t. the
empirical MSE of the i.i.d. twisted one, in the exponential case with $%
P_{n}=10^{-2}$ and $n=100$. The value of $k$ grows from $k=0$ (i.i.d.
twisted sample) to $k=70$ (according to the rule presented in \cite%
{BroniatowskiCaron2011IS}). This ratio stabilizes to $\sqrt{n-k}/\sqrt{n}$
for $L=2000$. The abscissa is $k$ and the solid line is $k\rightarrow
\sqrt{%
n-k}/\sqrt{n}$.

\begin{remark}
In the present context, Dupuis and Wang \cite{DupuisWang2004} have
shown that i.i.d.
sampling schemes can produce ``rogue paths'' which may alter the
properties of
the estimate, and the estimation of its variance. They consider an i.i.d.
random sample $X_{1}^{n}$ where $X_{1}$ has a normal distribution $N(1,1)$
and
\[
\mathcal{E}_{n}:= \biggl\{ x_{1}^{n}\dvtx
\frac{x_{1}+ \cdots+x_{n}}{n}\in A \biggr\},
\]
where $A= ( -\infty,a ) \cup ( b,+\infty ) $
with $%
a<1<b$. The quantity to be estimated is $P ( \mathcal{E}_{n} ) $.

Assuming that $a+b<2$, the standard i.i.d. IS scheme introduces the dominating
point $b$ and the family of i.i.d. tilted r.v.'s with common $N(b,1)$
distribution. ``Rogue paths'' generated under $N(b,1)$ may hit the set
$ (
-\infty,a ) $ with small probability under the sampling scheme,
hence producing a very large importance factor. The resulting variance
of the estimate is very sensitive with respect to these values, as
exemplified in their Table~1, page 24. Simulation of paths according to $G_{nS}$
with $S$ defined in (\ref{condSglobal}) produces their constructive
samples which yield both a hit rate close to 100\% and an importance
factor close to $P ( \mathcal{E}_{n} )$. We refer to \cite
{BroniatowskiCaron2011IS} for discussion and examples. We also note
that Dupis and Wang \cite{DupuisWang2004} propose an adaptive tilting
scheme, based on the product of the $\pi^{m_{i}}$, $1\leq i\leq n$,
which yields an efficient IS algorithm.
\end{remark}

\begin{appendix}
\section*{\texorpdfstring{Appendix}{Appendix}}\label{app}

For clarity the current term $a_{n}$ is denoted $a$ in all proofs.

\subsection{\texorpdfstring{Three lemmas pertaining to the
partial sum under its final value.}
{Three lemmas pertaining to the partial sum under its final value}}

We state three lemmas which describe some functions of the random
vector $%
\mathbf{X}_{1}^{n}$ conditioned on $\mathcal{E}_{n}$. The r.v. $\mathbf{X}$
is assumed to have expectation $0$ and variance $1$.

\begin{lemma}
\label{LemmaMomentsunderE_n}It holds that $E_{P_{ na }} ( \mathbf
{X}_{1} )
=a,E_{P_{ na }} ( \mathbf{X}_{1}\mathbf{X}_{2} ) =a{}^{2}+O (
\frac{1}{n} ) $,\break  $E_{P_{ na }} ( \mathbf{X}_{1}^{2} )
=s^{2}(t)+a{}^{2}+ O ( \frac{1}{n} ) $ where $m(t)=a$.
\end{lemma}

\begin{pf}
Using
\[
p_{ na }(\mathbf{X}_{1}=x)=\frac{p_{\mathbf{S}_{2,n}} (  na -x )
p_{%
\mathbf{X}_{1}}(x)}{p_{\mathbf{S}_{1,n}} (  na  ) }=
\frac{\pi_{%
\mathbf{S}_{2,n}}^{a} (  na -x ) \pi_{\mathbf
{X}_{1}}^{a}(x)}{\pi_{%
\mathbf{S}_{1,n}}^{a} (  na  ) },
\]
normalizing both $\pi_{\mathbf{S}_{2,n}}^{a} (  na -x ) $ and
$\pi_{%
\mathbf{S}_{1,n}}^{a} (  na  ) $ and making use of a first order
Edgeworth expansion in those expressions yields $E_{P_{ na }} (
\mathbf{X}%
_{1}^{2} ) =s^{2}(t)+a{}^{2}+ O ( \frac{1}{n} ) $. A~similar
expansion for the joint density $p_{ na }(\mathbf{X}_{1}=x,\mathbf
{X}_{2}=y)$%
, with the same tilted distribution $\pi^{a}$ produces the limit expression
of $E_{P_{ na }} ( \mathbf{X}_{1}\mathbf{X}_{2} ) $.
\end{pf}

\begin{lemma}
\label{LemmaMaxm_in} Assume \textup{(\ref{E1})}. Then \textup{(i)} $\max_{1\leq i\leq
k}\llvert  m_{i}\rrvert =a+o_{P_{ na }} ( \varepsilon_{n} ) $.
Also \textup{(ii)}~$\max_{1\leq i\leq k}s_{i}^{2}$, $\max_{1\leq i\leq k}\mu_{3}^{i}$
and $\max_{1\leq i\leq k}\mu_{4}^{i}$ tend in $P_{ na }$ probability to the
variance, skewness and kurtosis of $\pi^{\underline{a}}$ where
$\underline{a%
}:=\lim_{n\rightarrow\infty}a_{n}$.
\end{lemma}

\begin{pf}
(i) Define
\begin{eqnarray*}
V_{i+1}&:=&m(t_{i})-a
\\
& =&\frac{S_{i+1,n}}{n-i}-a.
\end{eqnarray*}
We state that
%
\begin{equation}
{\max_{0\leq{i\leq{k-1}}}|V_{i+1}|=}o_{P_{ na }} ( \varepsilon
_{n} ), \label{maxV_i+1}
\end{equation}
namely for all positive $\delta$
\[
\lim_{n\rightarrow\infty}P_{ na } \Bigl( \max_{0\leq{i\leq{k-1}}%
}|V_{i+1}|>
\delta\varepsilon_{n} \Bigr) =0,
\]
which we obtain following the proof of Kolmogorov maximal inequality. Define
\[
A_{i}:= \bigl( \bigl( |V_{i+1}|\geq\delta\varepsilon_{n}
\bigr) \mbox{ and }%
\bigl( |V_{j}|<\delta\varepsilon_{n}\mbox{
for all }j<i+1 \bigr) \bigr)
\]
from which
\[
\Bigl( \max_{0\leq{i}\leq{k-1}}|V_{i+1}|>\delta
\varepsilon_{n} \Bigr) =\bigcup_{i=0}^{k-1}A_{i}.
\]
It holds that
\begin{eqnarray*}
E_{P_{ na }}V_{k}^{2}& =&\int_{\cup A_{i}}V_{k}^{2}\,dP_{ na }+
\int_{ (
\cup
A_{i} ) ^{c}}V_{k}^{2}\,dP_{ na }
\\
& \geq&{\int_{\cup A_{i}}} \bigl( {V_{i}^{2}+2}
( V_{k}-V_{i} ) V_{i} \bigr) \,{dP_{ na }+
\int_{ ( \cup A_{i} ) ^{c}}} \bigl( {%
V_{i}^{2}+2}
( V_{k}-V_{i} ) V_{i} \bigr) \,{d}P_{ na }
\\
& \geq&\int_{\cup A_{i}}V_{i}^{2}\,dP_{ na }
\\
& \geq&\delta^{2}\varepsilon_{n}^{2}{\sum
_{j=0}^{k-1}P_{ na }(A_{j})}
\\
& =&\delta^{2}\varepsilon_{n}^{2}P_{ na }
\Bigl( {\max_{0\leq{i}\leq{k-1}%
}|V_{i+1}|>}\delta
\varepsilon_{n} \Bigr).
\end{eqnarray*}

The third line above follows from $EV_{i} ( V_{k}-V_{i} ) =0$ which
is proved below. Hence
\[
P_{ na } \Bigl( {\max_{0\leq{i}\leq{k-1}}|V_{i+1}|>}
\delta\varepsilon _{n} \Bigr) \leq\frac{\operatorname{Var}_{P_{ na }}(V_{k})}{\delta^{2}\varepsilon
_{n}^{2}}=%
\frac{1}{\delta^{2}\varepsilon_{n}^{2} ( n-k ) }\bigl(1+o(1)\bigr),
\]
where we used Lemma \ref{LemmaMomentsunderE_n}; therefore (\ref{maxV_i+1})
holds under (\ref{E1}). Direct calculation yields $E_{P_{ na }} (
V_{i}(V_{k}-V_{i}) ) ={0}$, which completes the proof of (i).

(ii) follows from (i) since $\lim_{n\rightarrow\infty}\max_{1\leq
i\leq
k}m(t_{i})=\underline{a}$.
\end{pf}

We also need the order of magnitude of $\max ( |\mathbf
{X}_{1}|,\ldots,|%
\mathbf{X}_{k}| ) $ under $P_{ na }$ which is stated in the following
result.

\begin{lemma}
\label{Lemma_max_X_i_under_conditioning} It holds that $\max (
\llvert
\mathbf{X}_{1}\rrvert,\ldots,\llvert \mathbf{X}_{n}\rrvert
)
=O_{P_{ na }} ( \log n ) $.
\end{lemma}

\begin{pf}
Set $\llvert \mathbf{X}_{1}\rrvert:=\mathbf{X}_{i}^{-}+\mathbf
{X}%
_{i}^{+}$ with $\mathbf{X}_{i}^{-}:=-\min ( 0,\mathbf{X}_{i}
) $,  $\mathbf{X}_{i}^{+}:=\break\max ( 0,\mathbf{X}_{i} ) $; it is
enough to
prove that $\max_{i}\mathbf{X}_{i}^{-}=O_{P_{ na }} ( \log n )$
and $\max_{i}\mathbf{X}_{i}^{+}=O_{P_{ na }} (\log n )$. Since
$E[\exp(t\mathbf{X})]$ is finite in a nonempty neighborhood of $0$ so
are $E[\exp(t\mathbf{X}^{-})]$ and $E[\exp(t\mathbf{X}^{+})]$. We hence
prove the lemma for positive r.v.'s $\mathbf{X}_{i}$ 's only.

Denote $a$ the current term of the sequence $a_{n}$. For all $t$ it
holds that
\begin{eqnarray*}
P_{ na } \bigl( \max ( \mathbf{X}_{1},\ldots,
\mathbf{X}_{n} ) >t \bigr) & \leq &nP_{ na } (
\mathbf{X}_{n}>t )
\\
& =&n\int_{t}^{\infty}\pi^{a} (
\mathbf{X}_{n}=u ) \frac{\pi
^{a}(%
\mathbf{S}_{1,n-1}= na -u)}{\pi^{a} ( \mathbf{S}_{1,n}= na  ) }\,du.
\end{eqnarray*}
Let $\tau$ be such that $m(\tau)=a$. Denote $s:=s(\tau)$. Center and
normalize both $\mathbf{S}_{1,n}$ and $\mathbf{S}_{1,n-1}$with respect to
the density $\pi^{a}$ in the last line above, denoting $\overline{\pi
_{n}^{a}}$ the density of $\overline{\mathbf{S}_{1,n}}:= ( \mathbf
{S}%
_{1,n}- na  ) /s\sqrt{n}$ when $\mathbf{X}$ has density $\pi^{a}$ with
mean $a$ and variance $s^{2}$, we obtain
\begin{eqnarray*}
&&P_{ na } \bigl( \max ( \mathbf{X}_{1},\ldots,
\mathbf{X}_{n} ) >t \bigr) \\
&&\qquad \leq n\frac{\sqrt{n}}{\sqrt{n-1}}\\
&&\qquad\quad{}\times\int
_{t}^{\infty}\pi^{a} ( \mathbf{X%
}_{n}=u )
 \frac{\overline{\pi_{n-1}^{a}} ( \overline{\mathbf
{S}_{1,n-1}}= (
 na -u-(n-1)a ) / ( s\sqrt{n-1} )  ) }{\overline{\pi
_{n}^{a}} ( \overline{\mathbf{S}_{1,n}}=0 ) }\,du.
\end{eqnarray*}
Under the sequence of densities $\pi^{a}$ the triangular array $ (
\mathbf{X}_{1},\ldots,\mathbf{X}_{n} ) $ obeys a first order Edgeworth
expansion
\begin{eqnarray*}
&&P_{ na } \bigl( \max ( \mathbf{X}_{1},\ldots,
\mathbf{X}_{n} ) >t \bigr)\\
&&\qquad \leq n\frac{\sqrt{n}}{\sqrt{n-1}}\int
_{t}^{\infty}\pi^{a} ( \mathbf{X%
}_{n}=u )
 \frac{\mathfrak{n} (  ( a-u ) /s\sqrt{n-1} ) \mathbf
{P}%
( u,i,n ) +o(1)}{\mathfrak{n} ( 0 ) +o(1)}\,du
\\
&&\qquad \leq n\mathit{Cst}\int_{t}^{\infty}\pi^{a} (
\mathbf{X}_{n}=u ) \,du
\end{eqnarray*}
for some constant $\mathit{Cst}$ independent of $n$ and $\tau$ and
\[
\mathbf{P} ( u,i,n ):=1+P_{3} \bigl( ( a-u ) /s\sqrt {n-1}%
\bigr),
\]
where $P_{3}(x)=\frac{\mu_{3}}{6s^{3}} ( x^{3}-3x ) $ is the third
Hermite polynomial; $s^{2}$ and $\mu_{3}$ are the second and third centered
moments of $\pi^{a}$. We have used the fact that the sequence $a$ converges
to bound all moments of the tilted densities $\pi^{a}$. We used
uniformity on $u$ in the remaining term of the Edgeworth expansions.
Making use of the Chernoff inequality to bound $\Pi^{a} ( \mathbf
{X}_{n}>t ) $,
\[
P_{ na } \bigl( \max ( \mathbf{X}_{1},\ldots,
\mathbf{X}_{n} ) >t \bigr) \leq n\mathit{Cst}\frac{\Phi(t+\lambda)}{\Phi(t)}e^{-\lambda t}
\]
for any $\lambda$ such that $\phi(t+\lambda)$ is finite. For $t$ such
that
\[
t/\log n\rightarrow\infty
\]
it holds that
\[
P_{ na } \bigl( \max ( \mathbf{X}_{1},\ldots,
\mathbf{X}_{n} ) <t \bigr) \rightarrow1,
\]
which proves the lemma.
\end{pf}

\subsection{\texorpdfstring{Proof of the approximations resulting
from Edgeworth expansions
in Theorem 2.}{Proof of the approximations resulting from Edgeworth expansions
in Theorem \ref{Thm:Approx_local_cond_density}}}

We complete the calculation leading to (\ref{Adem}) and (\ref{Bdem}).

Set $Z_{i+1}:= ( m_{i}-Y_{i+1} ) /s_{i}\sqrt{n-i-1}$.

It then holds that
%
\begin{eqnarray}\label{Hermite}
&&\overline{\pi_{n-i-1}} ( Z_{i+1} )\nonumber\\
 &&\qquad =\mathfrak
{n}(Z_{i+1}) \left[ \matrix{ \displaystyle 1+
\frac{1}{\sqrt{n-i-1}}P_{3}(Z_{i+1})+\frac{1}{n-i-1}P_{4}(Z_{i+1})
\vspace*{2pt}\cr
\displaystyle +\frac{1}{ ( n-i-1 ) ^{3/2}}P_{5}(Z_{i+1})%
}
\right]
\\
&&\qquad\quad{} +O_{P_{ na }} \biggl( \frac{P_{5}(Z_{i+1})}{ ( n-i-1 )
^{3/2}} \biggr).
\nonumber
\end{eqnarray}

We perform an expansion in $\mathfrak{n}(Z_{i+1})$ up to order $3$, with
a first order term $\mathfrak{n} ( -Y_{i+1}/ ( s_{i}\sqrt {n-i-1}%
)  )$, namely
%
\begin{eqnarray}\label{approx gauss}
&&\mathfrak{n}(Z_{i+1})\nonumber\\
&&\qquad =\mathfrak{n} \bigl( -Y_{i+1}/ (
s_{i}\sqrt {n-i-1}%
) \bigr)
\\
&&\qquad\quad{}\times \pmatrix{\displaystyle 1+\frac{Y_{i+1}m_{i}}{s_{i}^{2} ( n-i-1 ) }+
\frac{m_{i}^{2}}{%
2s_{i}^{2} ( n-i-1 ) } \biggl( \frac{Y_{i+1}^{2}}{s_{i}^{2} (
n-i-1 ) }-1 \biggr)
\vspace*{2pt}\cr
\displaystyle +\frac{m_{i}^{3}}{6s_{i}^{3} ( n-i-1 ) ^{3/2}}\frac{\mathfrak
{n}%
^{(3)} ( {Y^{\ast}}/{ ( s_{i}\sqrt{n-i-1} ) } ) }{
\mathfrak{n} ( -Y_{i+1}/ ( s_{i}\sqrt{n-i-1} )  ) }},\nonumber
\end{eqnarray}
where $Y^{\ast}=\frac{1}{s_{i}\sqrt{n-i-1}}(-Y_{i+1}+\theta m_{i})$
with $%
\llvert \theta\rrvert <1$.

Lemmas \ref{LemmaMaxm_in} and \ref{Lemma_max_X_i_under_conditioning} provide
the orders of magnitude of the random terms in the above displays when
sampling under $P_{ na }$.\vspace*{1pt}

Use those lemmas to obtain
%
\begin{equation}
\frac{Y_{i+1}m_{i}}{s_{i}^{2} ( n-i-1 ) }=\frac{Y_{i+1}}{n-i-1} 
\bigl( a+o_{P_{ na }} (
\varepsilon_{n} ) \bigr) \label{control 1}
\end{equation}
and
\[
\frac{m_{i}^{2}}{s_{i}^{2} ( n-i-1 ) }=\frac{1}{n-i-1} \bigl( a+o_{P_{ na }} (
\varepsilon_{n} ) \bigr) ^{2}.
\]
Also when (\ref{E1}) and (\ref{E2}) holds, then the dominant terms in
the bracket in (\ref{approx
gauss}) are precisely those in the two displays just above. This yields
%
\[
\mathfrak{n}(Z_{i+1})=\mathfrak{n} \biggl( \frac{-Y_{i+1}}{s_{i}\sqrt {n-i-1}}%
\biggr) \pmatrix{\displaystyle 1+\frac{aY_{i+1}}{s_{i}^{2}(n-i-1)}-
\frac{a^{2}}{2s_{i}^{2}(n-i-1)}
\vspace*{2pt}\cr
\displaystyle +\frac{o_{P_{ na }}(\varepsilon_{n}\log n)}{n-i-1}}.
\]

We now need a precise evaluation of the terms in the Hermite
polynomials in~(%
\ref{Hermite}). This is achieved using Lemmas \ref{LemmaMaxm_in} and
\ref%
{Lemma_max_X_i_under_conditioning} which provide uniformity on $i$ between
$1$ and $k=k_{n}$ in all terms depending on the sample path
$Y_{1}^{k}$. The
Hermite polynomials depend upon the moments of the underlying density $%
\pi^{m_{i}}$. Since $\overline{\pi_{1}^{m_{i}}}$ has expectation $0$ and
variance $1$ the terms corresponding to $P_{1}$  and $P_{2}$ vanish.
For up to order 4 polynomials, write $P_{3}(x)=\frac{\mu
_{3}^{(i)}}{6 (
s_{i} ) ^{3}}H_{3}(x)$, $P_{4}(x)=\frac{(\mu_{3}^{i})^{2}}{72 (
s_{i} ) ^{6}}H_{6}(x)+\frac{\mu_{4}^{(i,n)}-3 ( s_{i} )
^{4}}{%
24 ( s_{i} ) ^{4}}H_{4}(x) $ with $H_{3}(x):=x^{3}-3x$, $H_{4}(x):=
x^{4}-6x^{2}+3$ and $H_{6}(x):=x^{6}-15x^{4}+45x^{2}-15$.

Using Lemma \ref{LemmaMaxm_in} it appears that the terms in $x^{j}$,
$j\geq
3 $ in $P_{3}$ and $P_{4}$ will play no role in the asymptotic behavior
in (%
\ref{Hermite}) with respect to the constant term in $P_{4}$ and the
term in $%
x$ from $P_{3}$. Indeed substituting $x$ by $Z_{i+1}$ and dividing by $%
n-i-1$, the term in $x^{2}$ in $P_{4}$ is $O_{P_{ na }} ( \log
n ) ^{2}/(n-i)^{2}$ where we have used Lemma \ref{LemmaMaxm_in}.
These terms
are of smaller order than the term $-3x$ in $P_{3}$ which is $-\frac{\mu
_{3}^{i}}{2s_{i}^{4} ( n-i-1 ) } ( a-Y_{i+1} ) =\frac
{1}{%
n-i-1}O_{P_{ na }} ( \log n ) $.

It holds that
\begin{eqnarray*}
\frac{P_{3}(Z_{i+1})}{\sqrt{n-i-1}}& =&-\frac{\mu
_{3}^{i}}{2s_{i}^{4} (
n-i-1 ) } ( m_{i}-Y_{i+1} )
\\
&&{} +\frac{\mu_{3}^{i} ( m_{i}-Y_{i+1} ) ^{3}}{6 (
s_{i} ) ^{6}(n-i-1)^{2}},
\end{eqnarray*}
which yields
%
\begin{equation}
\label{P3} \frac{P_{3}(Z_{i+1})}{\sqrt{n-i-1}}=-\frac{\mu_{3}^{i}}{2s_{i}^{4} (
n-i-1 ) } ( a-Y_{i+1} )
+\frac{O_{P_{ na }} ( \log
n ) ^{3}}{ ( n-i-1 ) ^{2}}.
\end{equation}
For the term of order $4$ it holds that
\[
\frac{P_{4}(Z_{i+1})}{n-i-1}=\frac{1}{n-i-1} \biggl( \frac{(\mu
_{3}^{i})^{2}}{%
72s_{i}^{6}}H_{6}(Z_{i+1})+
\frac{\mu_{4}^{i}-3s_{i}^{4}}{24s_{i}^{4}}%
H_{4}(Z_{i+1}) \biggr),
\]
which yields
%
\begin{equation}
\label{P4} \frac{P_{4}(Z_{i+1})}{n-i-1}=\frac{\mu_{4}^{i}-3s_{i}^{4}}{%
8s_{i}^{4} ( n-i-1 ) }-\frac{15(\mu_{3}^{i})^{2}}{%
72s_{i}^{6}(n-i-1)}+
\frac{O_{P_{ na }} ( (\log n)^{2} ) }{ (
n-i-1 ) ^{2}}.
\end{equation}
The fifth term in the expansion plays no role in the asymptotics.

In summary, comparing the remainder terms in (\ref{P3}) and (\ref{P4}), we
obtain
\[
\overline{\pi_{n-i-1}} ( Z_{i+1} ) =\mathfrak{n} \bigl(
-Y_{i+1}/ ( s_{i}\sqrt{n-i-1} ) \bigr)\cdot A\cdot B+O_{P_{ na }}
\biggl( \frac{%
P_{5}(Z_{i+1})}{ ( n-i-1 ) ^{3/2}} \biggr),
\]
where $A$ and $B$ are given in (\ref{Adem}) and (\ref{Bdem}).

\subsection{\texorpdfstring{Final step of the proof of Theorem 2.}
{Final step of the proof of Theorem \ref{Thm:Approx_local_cond_density}}}

We make use of the following version of the law of large numbers for
triangular arrays; see \cite{Taylor1985} Theorem 3.1.3.

\begin{theorem}
\label{TheoremTaylor}Let $X_{i,n}$, $1\leq i\leq k$ denote an array of
row-wise real exchangeable r.v.'s and $\lim_{n\rightarrow\infty}k=\infty$.
Let $\rho_{n}:=EX_{1,n}X_{2,n}$. Assume that for some finite $\Gamma$,
$%
EX_{1,n}^{2}\leq\Gamma$. If for some doubly indexed sequence $ (
a_{i,n} ) $ such that $\lim_{n\rightarrow\infty}%
\sum_{i=1}^{k}a_{i,n}^{2}=0$ it holds that
\[
\lim_{n\rightarrow\infty}\rho_{n} \Biggl( \sum
_{i=1}^{k}a_{i,n}^{2} \Biggr)
^{2}=0
\]
and then
\[
\lim_{n\rightarrow\infty}\sum_{i=1}^{k}a_{i,n}X_{i,n}=0
\]
in probability.
\end{theorem}

Denote
\begin{eqnarray*}
\kappa_{1}^{i}&:=&\frac{\mu_{3}^{i}}{2s_{i}^{4}},\qquad \kappa_{2}^{i}:=
\frac{\mu_{4}^{i}-3s_{i}^{4}}{8s_{i}^{4}}+\frac{%
15(\mu_{3}^{i})^{2}}{72s_{i}^{6}},
\\
\mu_{1}^{\ast}&:=&\kappa_{1}^{i}+
\frac{a}{s_{i}^{2}}, \qquad\mu_{2}^{\ast}:=\kappa_{1}^{i}-
\frac{a}{2s_{i}^{2}}.
\end{eqnarray*}

By (\ref{condTilt}), (\ref{num approx fixed i}) and (\ref{PI 0})
\begin{eqnarray*}
&&p(\mathbf{X}_{i+1}=Y_{i+1}|S_{i+1,n}= na -S_{1,i})\\
&&\qquad=
\frac{\sqrt{n-i}}{\sqrt{n-i-1}}\pi^{m_{i}} ( \mathbf{X}%
_{i+1}=Y_{i+1} ) \frac{\mathfrak{n} (
{-Y_{i+1}}/{(s_{i}\sqrt{%
n-i-1})} ) }{\mathfrak{n}(0)}A(i)
\end{eqnarray*}
with
\begin{eqnarray*}
A(i)&:=&\biggl(1+\frac{\mu _{1}^{\ast }Y_{i+1}}{n-i-1}-
\frac{\mu _{2}^{\ast }a}{%
n-i-1}-\frac{\kappa _{2}^{i}}{n-i-1}+\frac{o_{P_{ na }}( \varepsilon
_{n}\log n) }{n-i-1}\biggr)\\
&&{}\bigg/\biggl(1-\frac{\kappa _{2}^{i}}{n-i}+O_{P_{ na }}\biggl(
\frac{1}{(n-i)^{3/2}}\biggr) \biggr).
\end{eqnarray*}
%
We perform a second order expansion in both the numerator and the
denominator of the above expression, which yields
%
\begin{eqnarray}\label{A(i)}
&&A(i)=\exp \biggl( \frac{\mu_{1}^{\ast}Y_{i+1}}{n-i-1}-\frac{a}{%
2s_{i}^{2}(n-i-1)}-\frac{a\kappa_{1}^{i}}{n-i-1}
\nonumber
\\[-8pt]
\\[-8pt]
\nonumber
&&\hspace*{162pt}{}+
\frac{o_{P_{ na }} (
\varepsilon_{n}\log n ) }{n-i-1} \biggr) A^{\prime}(i).
\end{eqnarray}

The term $\exp ( \frac{\mu_{1}^{\ast}Y_{i+1}}{n-i-1}+\frac{a}{%
2s_{i}^{2}(n-i-1)} ) $ in (\ref{A(i)}) is captured in $g(
Y_{i+1}\rrvert  Y_{1}^{i})$.

The term $A^{\prime}(i)$ in (\ref{A(i)}) is expressed as
\[
A^{\prime}(i):=Q_{1}^{i}\cdot Q_{2}^{i}
\]
with
\begin{eqnarray*}
&&Q_{1}^{i}:=\exp \biggl( - \biggl( %
 \frac{\kappa_{2}^{i}}{(n-i-1)(n-i)}+\frac{(\kappa
_{2}^{i})^{2}}{2(n-i)^{2}}\\
&&\hspace*{68pt}{}+%
\frac{1}{2} \biggl( \frac{\mu_{1}^{\ast}Y_{i+1}}{n-i-1}-\frac{a\mu
_{2}^{\ast}%
}{n-i-1}-
\frac{\kappa_{2}^{i}}{n-i-1} \biggr) ^{2}%
\biggr) \biggr)
\end{eqnarray*}
and
\[
Q_{2}^{i}:=\frac{\exp(B_{1})}{\exp(B_{2})},
\]
where
\begin{eqnarray*}
B_{1}&:=&\frac{o_{P_{ na }}(\varepsilon_{n}^{2}(\log
n)^{2})}{(n-i-1)^{2}}+\frac{%
\mu_{1}^{\ast}Y_{i+1}}{(n-i-1)^{2}}o_{P_{ na }} (
\varepsilon_{n}\log n )
\\
&&{} +\frac{\mu_{2}^{\ast}a}{(n-i-1)^{2}}o_{P_{ na }}(\varepsilon_{n}\log n)+
\frac{o_{P_{ na }}(\varepsilon_{n}^{2} ( \log n) )
^{2}}{(n-i-1)^{2}}%
+o\bigl(u_{1}^{2}
\bigr),
\\
B_{2} &:=&\frac{\kappa_{2}^{i}}{n-i}O_{P_{ na }}
\biggl( \frac
{1}{(n-i)^{3/2}}%
\biggr) +O_{P_{ na }} \biggl(
\frac{1}{(n-i)^{3}} \biggr)
\\
&&{} + O_{P_{ na }} \biggl( \frac{1}{(n-i)^{3/2}}
\biggr) +o \biggl( \biggl( \frac
{\kappa
_{2}^{i}}{n-i}+O_{P_{ na }} \biggl(
\frac{1}{(n-i)^{3/2}} \biggr) \biggr) ^{2} \biggr)
\end{eqnarray*}
with
\[
u_{1}=\frac{\mu_{1}^{\ast}Y_{i+1}}{n-i-1}-\frac{\mu_{2}^{\ast
}a}{n-i-1}-%
\frac{\kappa_{2}^{i}}{n-i-1}+\frac{o_{P_{ na }}(\varepsilon_{n}\log n)}{n-i-1}.
\]

We first prove that
%
\begin{equation}
\prod_{i=0}^{k-1}A^{\prime}(i)=1+o_{P_{ na }}
\bigl(\varepsilon_{n} ( \log n ) ^{2}\bigr) \label{cv produit des A'(i)}
\end{equation}
as $n$ tends to infinity.

Since
\[
p\bigl(\mathbf{X}_{1}^{k}=Y_{1}^{k}|S_{i+1}^{n}= na
\bigr)=g_{0} (  Y_{1}\rrvert Y_{0} )
\prod_{i=0}^{k-1}g \bigl(
Y_{i+1}\rrvert Y_{1}^{i} \bigr) \prod
_{i=0}^{k-1}A^{\prime
}(i)\prod
_{i=0}^{k-1}L_{i},
\]
where
\[
L_{i}:=\frac{C_{i}^{-1}}{\Phi ( t_{i} ) }\frac{\sqrt {n-i}}{\sqrt{%
n-i-1}}\exp \biggl( -
\frac{a\kappa_{1}^{i}}{n-i-1} \biggr),
\]
the completion of the proof will follow from
%
\begin{equation}
\prod_{i=0}^{k-1}L_{i}=1+o_{P_{ na }}
\bigl(\varepsilon_{n} ( \log n ) ^{2}\bigr). \label{cv produit des Li}
\end{equation}

The proof of (\ref{cv produit des A'(i)}) is achieved in two steps.

\begin{claim}
$\prod_{i=0}^{k-1}Q_{1}^{i}=1+o_{P_{ na }}(\varepsilon_{n} ( \log
n )
^{2})$.
\end{claim}

By Lemma \ref{LemmaMaxm_in} the random terms $\mu_{j}^{i}$ deriving
from $%
\pi^{m_{i}}$ satisfy
\[
\max_{1\leq i\leq k}\bigl\llvert \mu_{j}^{i}-
\mu_{j}\bigr\rrvert =o_{P_{ na }}(1)
\]
as $n$ tends to $\infty$, where $\mu_{j}$ is the $j$th cumulant of
$\pi
^{\underline{a}}$ where \underline{$a$}$:=\lim_{n\rightarrow\infty}a$ is
finite$.$ Therefore we may substitute $\mu_{j}^{i}$ by $\mu_{j}$ in order
to check the convergence of all subsequent series.

Expanding $Q_1$ define, for any positive $\beta_{1}$, $\beta_{2}$,
$\beta
_{3}$ and $\beta_{4}$
\begin{eqnarray*}
A_{n}^{1}&:= &\Biggl\{ \frac{1}{\varepsilon_{n} ( \log n ) ^{2}}%
\sum
_{i=0}^{k-1}\biggl\llvert \frac{\kappa_{2}^{i}}{(n-i-1)(n-i)}
\biggr\rrvert <\beta_{1} \Biggr\},
\\
A_{n}^{2}&:=& \Biggl\{ \frac{1}{\varepsilon_{n} ( \log n ) ^{2}}%
\sum
_{i=0}^{k-1}\biggl\llvert
\frac{(\kappa_{2}^{i})^{2}}{(n-i-1)^{2}}%
\biggr\rrvert <\beta_{2} \Biggr\},
\\
A_{n}^{3}&:=& \Biggl\{ \frac{1}{\varepsilon_{n} ( \log n ) ^{2}}%
\sum
_{i=0}^{k-1}\biggl\llvert
\frac{(\mu_{2}^{\ast}a)^{2}}{(n-i-1)^{2}}%
\biggr\rrvert <\beta_{3} \Biggr\}
\end{eqnarray*}
and
\[
A_{n}^{4}:= \Biggl\{ \frac{1}{\varepsilon_{n} ( \log n ) ^{2}}%
\sum
_{i=0}^{k-1}\biggl\llvert \frac{\mu_{2}^{\ast}\kappa_{2}^{i}a}{%
(n-i-1)^{2}}
\biggr\rrvert <\beta_{4} \Biggr\}.
\]

It clearly holds that
\[
\lim_{n\rightarrow\infty}P_{ na } \bigl( A_{n}^{j}
\bigr) =1, \qquad j=1,\ldots,4.
\]

Let for any positive $\beta_{5}$,
\[
A_{n}^{5}:= \Biggl\{ \frac{1}{\varepsilon_{n} ( \log n ) ^{2}}%
\sum
_{i=0}^{k-1}\biggl\llvert \frac{\kappa_{1}^{i}\kappa_{2}^{i}Y_{i+1}}{
(n-i-1)^{2}}
\biggr\rrvert <\beta_{5} \Biggr\}.
\]
If $\lim_{n\rightarrow\infty}P_{ na } ( A_{n}^{5} ) =1$, then $
\lim_{n\rightarrow\infty}P_{ na } ( A_{n}^{j} ) $, $j=6,7$ where
\begin{eqnarray*}
A_{n}^{6}&:=& \Biggl\{ \frac{1}{\varepsilon_{n} ( \log n ) ^{2}}%
\sum
_{i=0}^{k-1}\biggl\llvert \frac{\mu_{1}^{\ast}\kappa_{2}^{i}Y_{i+1}}{%
(n-i-1)^{2}}
\biggr\rrvert <\beta_{6} \Biggr\},
\\
A_{n}^{7}&:=& \Biggl\{ \frac{1}{\varepsilon_{n} ( \log n ) ^{2}}%
\sum
_{i=0}^{k-1}\biggl\llvert \frac{\mu_{1}^{\ast}\mu_{2}^{\ast}aY_{i+1}}{%
(n-i-1)^{2}}
\biggr\rrvert <\beta_{7} \Biggr\}.
\end{eqnarray*}

Apply Theorem \ref{TheoremTaylor} with $X_{i,n}=Y_{i+1}$ and
$a_{i,n}=\frac{%
1}{\varepsilon_{n} ( \log n ) ^{2}(n-i-1)^{2}}$. By Lem\-ma~\ref%
{LemmaMomentsunderE_n},
\[
E_{P_{ na }}Y_{1}^{2}=s^{2}(0)+a+O \biggl(
\frac{1}{n} \biggr).
\]
Hence $E_{P_{ na }}[Y_{1}^{2}]\leq{\Gamma}$ for some finite $\Gamma$.
Furthermore $\rho_{n}=a^{2}+O ( \frac{1}{n} ) $. Both
conditions in
Theorem \ref{TheoremTaylor} are fullfilled. Indeed,
\[
\lim_{n\rightarrow\infty}\sum_{i=1}^{k}a_{n,i}^{2}=
\lim_{n\rightarrow
\infty}%
\frac{1}{\varepsilon_{n}^{2} ( \log n ) ^{4}(n-k)^{3}}=0,
\]
which holds under (\ref{E1}), as holds
\[
\lim_{n\rightarrow\infty}\rho_{n} \Biggl( \sum
_{i=1}^{k}a_{n,i} \Biggr) ^{2}=
\lim_{n\rightarrow\infty}\frac{a^{2}}{\varepsilon_{n}^{2} ( \log
n ) ^{4}(n-k)^{2}}=0.
\]

Therefore, for $i=5,6,7$
\[
\lim_{n\rightarrow\infty}P_{ na } \bigl( A_{n}^{i}
\bigr) =1.
\]

Define for any positive $\beta_{8}$,
\[
A_{n}^{8}:= \Biggl\{ \frac{1}{\varepsilon_{n} ( \log n ) ^{2}}%
\sum
_{i=0}^{k-1}\frac{ ( \mu_{1}^{\ast} ) ^{2}Y_{i+1}^{2}}{%
(n-i-1)^{2}}<
\beta_{8} \Biggr\}.
\]

Apply Theorem \ref{TheoremTaylor} with $X_{i,n}=Y_{i+1}^{2}$ and
$a_{i,n}=%
\frac{1}{\varepsilon_{n} ( \log n ) ^{2}(n-i-1)^{2}}$.

The following holds:
\[
\lim_{n\rightarrow\infty}\sum_{i=1}^{k}a_{n,i}^{2}=0
\]
when (\ref{E1}) holds.

By Lemma \ref{LemmaMomentsunderE_n},
\[
E_{P_{ na }}Y_{1}^{4}=E_{\pi^{a}}Y_{1}^{4}+O
\biggl( \frac{1}{n} \biggr),
\]
which entails that such that $EY_{1}^{4}\leq{\Gamma}<\infty$ for
some $%
\Gamma$. Also
\[
E_{P_{ na }} \bigl( Y_{1}^{2}Y_{2}^{2}
\bigr) = \bigl( s^{2}(0)+a \bigr) \bigl( s^{2}(0)+a \bigr) +O
\biggl( \frac{1}{n} \biggr)
\]
and
\[
\lim_{n\rightarrow\infty}\rho_{n} \Biggl( \frac{1}{\varepsilon_{n} (
\log
n ) ^{2}}\sum
_{i=0}^{k-1}\frac{1}{(n-i-1)^{2}} \Biggr)
^{2}=0
\]
under (\ref{E1}). Hence
\[
\lim_{n\rightarrow\infty}P_{ na } \bigl( A_{n}^{8}
\bigr) =1.
\]
It follows that, noting that $A_{n}$ is the intersection of the events
$A_{n}^{i}$,
$j=1,\ldots,8$
\[
\lim_{n\rightarrow\infty}P_{ na } ( A_{n} ) =1.
\]
To summarize, we have proved that, under (\ref{E1}),
\[
Q_1=1+o_{P_{ na }} \bigl( \varepsilon_{n} ( \log n )
^{2} \bigr).
\]

\begin{claim}
$\prod_{i=0}^{k-1}Q_{2}^{i}=1+o_{P_{ na }} ( \varepsilon_{n} ( \log
n ) ^{2} ) $.
\end{claim}

This is equivalent to proving that the sum of the terms in $B_1$ (resp.,
in $B_2$) is of order $o_{P_{ na }} ( \varepsilon_{n} ( \log n
) ^{2} ) $.

The four terms in the sum of the terms in $B_1$ are, respectively, of order
$o_{P_{ na }} ( \varepsilon_{n}^{2}(\log n)^{4} ) /(n-k)$, $%
o_{P_{ na }} ( \varepsilon_{n}(\log n)^{3} ) /(n-k)$,
$o_{P_{ na }} (
a\varepsilon_{n}(\log n)^{2} ) /(n-k)$ and $o_{P_{ na }} ( \varepsilon
_{n}(\log n)^{2} ) /(n-k)$ using Lemma \ref{LemmaMaxm_in}. The sum of
the terms $o ( u_{1}^{2} ) $ is of order less than these.
Assuming (\ref{E1}) all these terms are $o_{P_{ na }} ( \varepsilon
_{n} ( \log n ) ^{2} ) $.

For the sum of terms in $B_2$, by uniformity of the Edgeworth expansion with
respect to $Y_{1}^{k}$ it holds that $\sum_{i=1}^{k} B_2=O_{P_{ na }}
(  (
n-k ) ^{-1/2} )=\break o_{P_{ na }} (\varepsilon_{n} (\log
n )^{2} ) $ by~(\ref{E1}).

We now turn to the proof of (\ref{cv produit des Li}).

Define
\[
u:=-x\frac{\mu_{3}^{i}}{2s_{i}^{4} ( n-i-1 ) }+\frac
{(x-a)^{2}}{%
2s_{i}^{2} ( n-i-1 ) }.
\]
Use the classical bounds
\[
1-u+\frac{u^{2}}{2}-\frac{u^{3}}{6}\leq e^{-u}\leq1-u+
\frac{u^{2}}{2}
\]
to obtain on both sides of the above inequalities the second order
approximation of $C_{i}^{-1}$ through integration with respect to $p$. The
upper bound yields
\begin{eqnarray*}
C_{i}^{-1}& \leq&\Phi(t_{i})+\frac{\kappa_{1}^{i}}{n-i-1}
\Phi^{\prime
}(t_{i})+\frac{1}{s_{i}^{2}(n-i-1)} \bigl(
\Phi^{\prime\prime
}(t_{i})-2a\Phi^{\prime
} ( t_{i} )
+a^{2} \bigr)
\\
& &{}+O_{P_{ na }} \biggl( \frac{1}{(n-i-1)^{2}} \biggr)
\end{eqnarray*}
from which
\begin{eqnarray*}
L_{i}&\leq&\frac{\sqrt{n-i}}{\sqrt{n-i-1}}\exp \biggl( -\frac{a\kappa
_{1}^{i}%
}{n-i-1} \biggr)\\
&&{}\times
\pmatrix{\displaystyle 1+\frac{\kappa_{1}^{i}}{n-i-1}m_{i}
\vspace*{2pt}\cr
\displaystyle -\frac{s_{i}^{2}+m_{i}^{2}-2am_{i}+a^{2}}{2s_{i}^{2} ( n-i-1 )
}%
+O_{P_{ na }} \biggl( \frac{1}{ ( n-i-1 ) ^{2}}
\biggr)},
\end{eqnarray*}
where the approximation term is uniform on the $Y_{1}^{k}$.

Substituting $\frac{\sqrt{n-i}}{\sqrt{n-i-1}}$ and $\exp ( -\frac{%
a\kappa_{1}^{i}}{n-i-1} ) $ by their expansions $1+\frac{1}{2 (
n-i-1 ) }+O ( \frac{1}{ ( n-i-1 ) ^{2}} ) $ and
$1-%
\frac{a\kappa_{1}^{i}}{n-i-1}+\frac{(a\kappa_{1}^{i})^{2}}{(n-i-1)^{2}}%
+O ( \frac{a^{2}}{(n-i-1)^{2}} ) $ in the upper bound of
$L_{i}$  above yields
\begin{eqnarray*}
L_{i}& \leq& \biggl( 1+\frac{1}{2(n-i-1)}-\frac{a\kappa
_{1}^{i}}{n-i-1}+%
\frac{(a\kappa_{1}^{i})^{2}}{2(n-i-1)^{2}}+o \biggl( \frac
{1}{(n-i-1)^{2}}%
\biggr) \biggr)
\\
&&{}\times \biggl( 1+\frac{\kappa_{1}^{i}m_{i}}{n-i-1}-\frac{%
s_{i}^{2}+m_{i}^{2}-2am_{i}+a^{2}}{2s_{i}^{2}(n-i-1)}+O_{P_{ na }} \biggl(
\frac{%
1}{(n-i-1)^{2}} \biggr) \biggr).
\end{eqnarray*}

Using Lemma \ref{LemmaMaxm_in}, $m_{i}^{2}-2am_{i}+a^{2}=o_{P_{ na }}(a%
\varepsilon_{n})$ and therefore
\begin{eqnarray*}
L_{i}& \leq& \biggl( 1+\frac{1}{2(n-i-1)}-\frac{a\kappa
_{1}^{i}}{n-i-1}+%
\frac{(a\kappa_{1}^{i})^{2}}{(n-i-1)^{2}}+o \biggl( \frac{1}{(n-i-1)^{2}} 
\biggr) \biggr)
\\
&&{}\times \biggl( 1+\frac{\kappa_{1}^{i}a}{n-i-1}-\frac{1}{2(n-i-1)}+\frac{%
o_{P_{ na }}(a\varepsilon_{n})}{n-i-1} \biggr)
.
\end{eqnarray*}
Write
\[
\prod_{i=1}^{k}L_{i}\leq{\prod
_{i=1}^{k}} ( {1+M_{i}} )
\]
with
\[
M_{i}=\frac{(a\kappa_{1}^{i})^{2}}{(n-i-1)^{2}}+\frac
{o_{P_{ na }}(a\varepsilon
_{n})}{n-i-1}.
\]

Under (\ref{E1}), $\sum_{i=0}^{k-1}M_{i}$ is $o_{P_{ na }} ( \varepsilon
_{n} ( \log n ) ^{2} ) $. This completes the proof of the theorem.

\subsection{\texorpdfstring{Proof of
Theorem \protect\ref{Thm_approx_largeSets}.}
{Proof of Theorem 18}}

The following lemma (see \cite{Jensen1995}, Corollary 6.4.1) provides an
asymptotic formula for the tail probability of $\mathbf{U}_{1,n}$ under the
hypotheses and notation of Section~\ref{sec3}. Define
\[
I_{\mathbf{U}}(x):=xm^{-1} ( x ) -\log\phi_{\mathbf{U}} \bigl(
m^{-1} ( x ) \bigr).
\]

\begin{lemma}
\label{Lemma_Jensen} Under the same hypotheses as above,
\[
P \biggl( \frac{\mathbf{U}_{1,n}}{n}>a \biggr) =\frac{\exp(-nI_{\mathbf
{U}}(a))}{\sqrt{2\pi}\sqrt{n}\psi(a)} \biggl(1+O \biggl(
\frac{1}{\sqrt {n}} \biggr) \biggr),
\]
where $\psi(a):=t^{a}s(t^{a})$.
\end{lemma}

\begin{lemma}
\label{LemmaMomentsCond} Suppose that \textup{(V)} holds. Then \textup{(i)}
$E_{P_{nA}}\mathbf{U%
}_{1}=a+o(1)$, \textup{(ii)} $E_{P_{nA}}\mathbf{U}_{1}^{2}=1+s^{2} ( t )
+o(1)$ and \textup{(iii)} $E_{P_{nA}}\mathbf{U}_{1}\mathbf{U}_{2}= a^{2}+o(1)$
where $m(t)=a$.
\end{lemma}

\begin{pf}
It holds that
\[
E_{P_{nA}}\mathbf{U}_{1}=\int_{a}^{\infty}
( E_{P_{nv}}\mathbf{U}%
_{1} ) p (
\mathbf{U}_{1,n}/n=v\rrvert \mathbf{U}%
_{1,n}> na  )
\,dv.
\]
Integration by parts yields
\[
E_{P_{nA}}\mathbf{U}_{1}=a+\int_{a}^{\infty}P
(  \mathbf{U}%
_{1,n}/n>v\rrvert \mathbf{U}_{1,n}> na
) \,dv.
\]
Using Lemma \ref{Lemma_Jensen} and the Chernoff inequality,
\begin{eqnarray*}
&&\int_{a}^{\infty}P (  \mathbf{U}_{1,n}/n>v
\rrvert \mathbf {U}%
_{1,n}> na  ) \,dv\\
&&\qquad\leq{\sqrt{2\pi}\psi(a)
\sqrt{n}\int_{a}^{\infty
}\exp \bigl(n \bigl(
I_{\mathbf{U}}(a)-I_{\mathbf{U}}(v) \bigr) \bigr)\,dv},
\end{eqnarray*}
where $\psi(a)=ts ( t ) $.

Finally, using $I_{\mathbf{U}}(v)>I_{\mathbf{U}}^{^{\prime
}}(a)v+I_{\mathbf{%
U}}(a)-aI_{\mathbf{U}}^{^{\prime}}(a)$ and integrating
\[
\int_{a}^{\infty}P (  \mathbf{U}_{1,n}/n>v
\rrvert \mathbf {U}%
_{1,n}> na  ) \,dv\leq{\frac{\sqrt{2\pi}\psi(a)}{\sqrt{n}I_{\mathbf
{U}%
}^{^{\prime}}(a)}.}
\]
Hence, $E_{P_{nA}}\mathbf{U}_{1}=a+o(1)$.

Insert $E_{P_{nv}}\mathbf{U}_{1}^{2}=v^{2}+s_{\mathbf{U}}^{2} (
t )
+O ( \frac{1}{n} ) $ into
\[
E_{p_{nA}}\mathbf{U}_{1}^{2}=\int
_{a}^{\infty} \bigl( E_{P_{nv}}\mathbf
{U}%
_{1}^{2} \bigr) p (
\mathbf{U}_{1,n}/n=v\rrvert \mathbf {U}%
_{1,n}> na  )
\,dv.
\]
First, via integration by parts, Lemma \ref{LemmaJensen} and the Chernoff
inequality,
\[
\int_{a}^{\infty}v^{2}p (
\mathbf{U}_{1,n}/n=v\rrvert \mathbf{U}_{1,n}> na  )
\,dv=a^{2}+o(1).
\]

Second,
\begin{eqnarray*}
&&\int_{a}^{\infty}V(v)p (  \mathbf{U}_{1,n}/n=v
\rrvert \mathbf{%
U}_{1,n}> na  ) \,dv\\
&&\qquad=
s^{2}(t)+2\int_{a}^{\infty}V^{^{\prime}}(v)P
(  \mathbf{U}%
_{1,n}/n>v\rrvert \mathbf{U}_{1,n}> na
) \,dv,
\end{eqnarray*}
which tends to $s^{2}(t)$ as $n\rightarrow\infty$ using again the Chernoff
inequality, condition (V) and Lemma \ref{LemmaJensen}.

The third term is handled similarly due to the fact that the $O(1/n)$
term consists of a sum of powers of $v$.

The proof of (iii) is similar to the above.
\end{pf}

Lemma \ref{LemmaMomentsCond} yields the maximal inequality stated in
Lemma %
\ref{LemmaMaxm_in} under the condition $ (\mathbf{U}_{1,n}> na  )$.
We also need the order of magnitude of the maximum of $ ( \llvert
\mathbf{U}_{1}\rrvert, \dots, \llvert \mathbf{U}_{k}\rrvert
) $
under $P_{nA}$ which is stated in the following result.

\begin{lemma}
\label{Lemma_max_U_i_under_E_n} It holds that
\[
\max \bigl( \llvert \mathbf{U}_{1}\rrvert, \dots,\llvert
\mathbf{U}%
_{n}\rrvert \bigr) =O_{P_{nA}}(\log n).
\]
\end{lemma}

\begin{pf}
Using the same argument as in Lemma \ref{Lemma_max_X_i_under_conditioning}
we consider the case when the r.v.'s $\mathbf{U}_{i}$ take nonnegative
values. We prove that
\[
\lim_{n\rightarrow\infty}P_{nA} \bigl( \max (
\mathbf{U}_{1},\ldots,%
\mathbf{U}_{n} )
>t_{n} \bigr) =0
\]
when
\[
\lim_{n\rightarrow\infty}\frac{t_{n}}{\log n}=\infty.
\]
For fixed $d$ it holds that
\begin{eqnarray*}
&&P_{nA} \bigl( \max ( \mathbf{U}_{1},\ldots,
\mathbf{U}_{n} ) >t_{n} \bigr) \\
&&\qquad=\int_{a}^{a+d}P
\bigl(  \max ( \mathbf {U}_{1},\ldots,%
\mathbf{U}_{n} ) >t_{n}\rrvert \mathbf{U}_{1,n}/n=v
\bigr)\\
&&\hspace*{59pt}{}\times p (  \mathbf{U}_{1,n}/n=v\rrvert \mathbf {U}_{1,n}/n>a
) \,dv
\\
&&\qquad\quad{}+\int_{a+d}^{\infty}P \bigl(  \max ( \mathbf
{U}_{1},\ldots,\mathbf{%
U}_{n} ) >t_{n}
\rrvert \mathbf{U}_{1,n}/n=v \bigr)\\
&&\hspace*{68pt}{}\times p (  \mathbf{U}_{1,n}/n=v\rrvert \mathbf {U}_{1,n}/n>a
) \,dv
\\
&&\qquad=:I+\mathit{II}.
\end{eqnarray*}
Now
\[
\mathit{II}\leq\frac{P ( \mathbf{U}_{1,n}/n>a+d ) }{P ( \mathbf{U}%
_{1,n}/n>a ) },
\]
which tends to $0$ by Lemma \ref{Lemma_Jensen}.

Furthermore by Lemma \ref{Lemma_max_X_i_under_conditioning}, $%
\lim_{n\rightarrow\infty} P (  \max ( \mathbf
{U}_{1},\ldots,%
\mathbf{U}_{n} ) >t_{n}\rrvert \mathbf{U}_{1,n}/n=v )
=:\lim_{n\rightarrow\infty}r_{n}=0$ when $v\in ( a,a+d )$.
Hence
\[
I\leq r_{n}\bigl(1+o(1)\bigr)\rightarrow0.
\]
This proves the lemma.
\end{pf}

We now prove (\ref{Thm_approx_largeSets(i)}).

\textit{Step} 1. We first prove that the integral (\ref{etoile}) can be
reduced to its principal part, namely that
\begin{eqnarray}\label{Reductiona+c}
\qquad p_{nA}\bigl(Y_{1}^{k}\bigr)&=& \bigl(
1+o_{P_{nA}} ( 1 ) \bigr)
\nonumber
\\[-8pt]
\\[-8pt]
\nonumber
&&{}\times \int_{a}^{a+c}p\bigl( \mathbf{X}_{1}^{k}=Y_{1}^{k}
\rrvert \mathbf {U}%
_{1,n}/n=v\bigr)p(
\mathbf{U}_{1,n}/n=v\rrvert \mathbf{U}_{1,n}> na )\,dv
\end{eqnarray}
holds for any fixed $c>0$.

Apply Bayes's formula to obtain
\begin{eqnarray*}
p_{nA}\bigl(Y_{1}^{k}\bigr)&=&\frac{np_{\mathbf{X}} ( Y_{1}^{k} ) }{ (
n-k ) }
\\
&&{}\times\frac{\int_{a}^{\infty}p (
{\mathbf{U}_{k+1,n}}/{(n-k)}=
{n}/{(n-k)}%
( t-{k\overline{U_{1,k}}}/{n} )  ) \,dt}{P (
\mathbf{U}%
_{1,n}> na  ) },
\end{eqnarray*}
where $\overline{U_{1,k}}:=\frac{U_{1,k}}{k}$.

Denote
\[
I:=\frac{P ( {\mathbf{U}_{k+1,n}}/{(n-k)}>m_{k}+
{nc}/{(n-k)} )
}{P ( {\mathbf{U}_{k+1,n}}/{(n-k)}>m_{k} ) }
\]
with
\[
m_{k}:=\frac{n}{n-k} \biggl( a-\frac{k\overline{U_{1,k}}}{n} \biggr).
\]
Then (\ref{Reductiona+c}) holds whenever $I\rightarrow0$ (under $P_{nA}$).

Under $P_{nA}$ it holds that
\[
\overline{U_{1,n}}=a+O_{P_{nA}} \biggl( \frac{1}{nm^{-1}(a)}
\biggr).
\]
A similar result as Lemma \ref{LemmaMaxm_in} holds under condition
$ (
\mathbf{U}_{1,n}> na  ) $, using Lemma~\ref{LemmaMomentsCond};
namely it
holds that
\[
\max_{0\leq i\leq k-1}\llvert \overline{U_{i+1,n}}\rrvert
=a+o_{P_{nA}} ( \varepsilon_{n} ).
\]
Using both results
%
\begin{equation}
m_{k}=a+O_{P_{nA}} ( v_{n} ) \label{inter1}
\end{equation}
with $v_{n}=\max ( \varepsilon_{n},\frac{1}{ ( n-k )
m^{-1}(a)}%
) $ which tends to $0$.

We now prove that $I\rightarrow0$. Using once more Lemma \ref{Lemma_Jensen}
yields
\begin{eqnarray*}
I&=&\frac{m^{-1} ( m_{k} ) s ( m^{-1} ( m_{k} )
)
}{m^{-1} ( m_{k}+{nc}/{(n-k)} ) s ( m^{-1} (
m_{k}+{nc}/{(n-k)} )  ) }
\\
&&{}\times\exp \biggl( - ( n-k ) \biggl( I_{\mathbf{U}} \biggl( m_{k}+
\frac
{nc}{n-k}%
\biggr) -I_{\mathbf{U}} ( m_{k} )
\biggr) \biggr).
\end{eqnarray*}

Now by convexity of the function $I_{\mathbf{U}}$
\begin{eqnarray*}
&&\exp \biggl(- ( n-k ) \biggl( I_{\mathbf{U}} \biggl( m_{k}+
\frac
{nc}{n-k}%
\biggr) -I_{\mathbf{U}} ( m_{k} )
\biggr) \biggr)
\\
&&\qquad\leq\exp \bigl(-ncm^{-1}(m_{k}) \bigr)\\
&&\qquad=\exp \biggl(-nc
\biggl[ m^{-1}(a)+\frac{1}{V(a+\theta
O_{P_{nA}}(v_{n}))}O_{P_{nA}}(v_{n})
\biggr] \biggr)
\end{eqnarray*}
for some $\theta$ in $ ( 0,1 ) $. Therefore the above upper bound
tends to $0$ under $P_{nA}$ when (C) holds. By monotonicity of $%
t\rightarrow m(t)$ and condition (C) the ratio in $I$ is bounded.

We have proved that
\[
I=O_{P_{nA}} \bigl( \exp(-nc) \bigr).
\]

\textit{Step} 2.  We claim that (\ref{Thm_approx_largeSets(i)}) holds
uniformly in $v$ in $ ( a,a+c ) $ when $Y_{1}^{k}$ is generated
under $P_{nA}$. This result follows from a similar argument as used in
Theorem \ref{ThmApproxsousf(x)} where (\ref{Thm_approx_largeSets(i)}) is
proved under the local sampling $P_{nv}$. A close look at the proof shows
that (\ref{Thm_approx_largeSets(i)}) holds whenever Lemmas \ref{LemmaMaxm_in}
and \ref{Lemma_max_X_i_under_conditioning}, stated for the variables $%
\mathbf{U}_{i}$'s instead of $\mathbf{X}_{i}$'s hold under $P_{nA}$. Those
lemmas are substituted by Lemmas \ref{LemmaMomentsCond} and \ref%
{Lemma_max_U_i_under_E_n} here above.

Inserting (\ref{Thm_approx_largeSets(i)}) in (\ref{Reductiona+c}) yields
\begin{eqnarray*}
p_{nA}\bigl(Y_{1}^{k}\bigr)& =& \biggl( \int
_{a}^{a+c}g_{nv}\bigl(Y_{1}^{k}
\bigr)p( \mathbf{U}%
_{1,n}/n=v\vert\mathbf{U}_{1,n}> na )\,dv
\biggr)
\\
&&{}\times \bigl( 1+o_{p_{nA}} \bigl( \max \bigl( \varepsilon_{n} ( \log n
) ^{2}, \bigl( \exp(-nc) \bigr) ^{\delta} \bigr) \bigr) \bigr)
\end{eqnarray*}
for some $\delta<1$.

The conditional density of $\mathbf{U}_{1,n}/n$ given $ ( \mathbf{U}
_{1,n}> na  ) $ is stated in (\ref{Approx_Exp}) which holds
uniformly in $%
v $ on $(a,a+c)$.

In summary we have proved
\begin{eqnarray*}
p_{nA}\bigl(Y_{1}^{k}\bigr)
&=& \biggl( nm^{-1} ( a ) \int_{a}^{a+c}g_{nv}
\bigl(Y_{1}^{k}\bigr)\exp \bigl( -nm^{-1} ( a ) (
v-a ) \bigr) \,dv \biggr)
\\
&&{}\times \bigl( 1+o_{p_{nA}} \bigl( \max \bigl( \varepsilon_{n} ( \log n
) ^{2}, \bigl( \exp(-nc) \bigr) ^{\delta} \bigr) \bigr) \bigr)
\end{eqnarray*}
as $n\rightarrow\infty$ for any positive $\delta<1$.

In order to obtain the approximation of $p_{nA}$ by the density
$g_{nA}$ it is
enough to observe that
\begin{eqnarray*}
&&nm^{-1} ( a ) \int_{a}^{a+c}g_{nv}
\bigl(Y_{1}^{k}\bigr)\exp \bigl( -nm^{-1} ( a ) (
v-a ) \bigr) \,dv
\\
&&\qquad=1+o_{_{P_{nA}}} \bigl( \exp(-nc) \bigr)
\end{eqnarray*}
as $n\rightarrow\infty$ which completes the proof of (\ref%
{Thm_approx_largeSets(i)}). The proof of (\ref{Thm_approx_largeSets(ii)})
follows from (\ref{Thm_approx_largeSets(i)}) and Lemma \ref%
{Lemma:commute_from_p_n_to_g_n}.
\end{appendix}

\section*{\texorpdfstring{Acknowledgements.}{Acknowledgements}}
The authors thank the referee for his careful reading of the paper and
for comments that considerably improved the presentation of this work.
Also the authors thank Dr. Tarn Duong for his help and discussions.

%



\printaddresses

\end{document}